\documentclass[journal]{IEEEtran}
\ifCLASSINFOpdf
\else
\fi

\usepackage{amsmath}
\usepackage{graphicx}
\usepackage{amssymb}
\usepackage{epstopdf}
\usepackage{amsthm}
\usepackage{bm}
\usepackage{algorithm}
\usepackage{cite}
\usepackage[noend]{algpseudocode}
\usepackage{soul,color}
\usepackage{mathtools}
\usepackage{caption}
\usepackage{subcaption}
\usepackage{array}

\DeclarePairedDelimiter{\norm}{\lVert}{\rVert}
\DeclarePairedDelimiter{\abs}{\lvert}{\rvert}

\def\P{{\cal P}}

\def\R{{\mathbb R}}
\def\O{{\cal O}}

\newcommand{\topnew}{\top \hspace{-0.05cm}}

\newcommand\numberthis{\addtocounter{equation}{1}\tag{\theequation}}

\algnewcommand{\IfThenElse}[3]{% \IfThenElse{<if>}{<then>}{<else>}
  \State \algorithmicif\ #1\ \algorithmicthen\ #2\ \algorithmicelse\ #3}

\DeclareMathOperator*{\vect}{vec}
\DeclareMathOperator*{\rank}{rank}
\DeclareMathOperator*{\tr}{tr}
\DeclareMathOperator*{\Var}{Var}
\DeclareMathOperator*{\Cov}{Cov}
\DeclareMathOperator*{\diag}{diag}

\newtheorem{assumption}{Assumption}
\newtheorem{theorem}{Theorem}
\newtheorem{conjecture}{Conjecture}
\newtheorem{example}{Example}
\newtheorem{proposition}{Proposition}
\newtheorem{lemma}{Lemma}
\newtheorem{remark}{Remark}

\newtheorem{definition}{Definition}

\makeatletter
\newcommand\fs@betterruled{%
 \def\@fs@cfont{\bfseries}\let\@fs@capt\floatc@ruled
 \def\@fs@pre{\vspace*{5pt}\hrule height.8pt depth0pt \kern2pt}%
 \def\@fs@post{\kern2pt\hrule\relax}%
 \def\@fs@mid{\kern2pt\hrule\kern2pt}%
 \let\@fs@iftopcapt\iftrue}
\floatstyle{betterruled}
\restylefloat{algorithm}
\makeatother

\begin{document}
\title{On Asymptotic Linear Convergence Rate of Iterative Hard Thresholding for Matrix Completion}

\author{Trung Vu,~\IEEEmembership{Graduate Student Member,~IEEE,}
        Evgenia~Chunikhina,~and~Raviv Raich,~\IEEEmembership{Senior Member,~IEEE}
        % <-this % stops a space
\thanks{Trung Vu and Raviv Raich are with the School of Electrical Engineering and Computer Science, Oregon State University, Corvallis, OR 97331, USA.}
\thanks{Evgenia~Chunikhina is with Department of Mathematics and Computer Science, Pacific University, Forest Grove, OR 97116, USA.}
\thanks{E-mails: vutru@oregonstate.edu; chunikhina@pacificu.edu; and raich@oregonstate.edu.}% <-this % stops a space
\thanks{This work was partially supported by the National Science Foundation grants CCF-1254218.}% <-this % stops a space
\thanks{Manuscript received 2021.}}

% Spaces after \IEEEmembership other than the last one are Oc_2 (and needed) as

\markboth{Preprint~2021}%
{Trung \MakeLowercase{\textit{et al.}}: On Asymptotic Linear Convergence Rate of Iterative Hard Thresholding for Matrix Completion}



\maketitle

\begin{abstract}
Iterative hard thresholding (IHT) has gained in popularity over the past decades in large-scale optimization. However, convergence properties of this method have only been explored recently in non-convex settings. In matrix completion, existing works often focus on the guarantee of global convergence of IHT via standard assumptions such as incoherence property and uniform sampling. While such analysis provides a global upper bound on the linear convergence rate, it does not describe the actual performance of IHT in practice. In this paper, we provide a novel insight into the local convergence of a specific variant of IHT for matrix completion. We uncover the exact asymptotic linear rate of IHT in a closed-form expression and identify the region of convergence in which the algorithm is guaranteed to converge. Furthermore, we utilize random matrix theory to study the linear rate of convergence of IHTSVD for large-scale matrix completion. We find that asymptotically, the rate can be expressed explicitly in terms of the relative rank and the sampling rate. Finally, we present numerical results to verify the foregoing theoretical analysis.
\end{abstract}

% Note that keywords are not normally used for peer-review papers.

% \fi
\IEEEpeerreviewmaketitle

\section{Introduction}

% \hl{TV: Introduce matrix completion problem}
\IEEEPARstart{M}{atrix} completion is a fundamental problem that arises in many areas of signal processing and machine learning such as collaborative filtering \cite{srebro2003weighted,srebro2005maximum,rennie2005fast,takacs2008investigation}, system identification \cite{liu2009interior,mohan2010reweighted,liu2013nuclear} and dimension reduction \cite{wright2009robust,candes2011robust}. 
The problem can be explained as follows. Let $\bm M \in \mathbb{R}^{n_1 \times n_2}$ be the underlying matrix with rank $r$ and $\Omega$ be the set of locations corresponding to the observed entries of $\bm M$, i.e., $(i,j) \in \Omega$ if $M_{ij}$ is observed.
The goal is to recover the unknown entries of $\bm M$, belonging to the complement set $\bar{\Omega}$. 

% \hl{TV: Feasibility of MCP}
To understand the feasibility of matrix completion, let us describe $\bm M$ as
\begin{align*}
    \bm M = \sum_{i=1}^r \sigma_i \bm u_i \bm v_i^\topnew ,
\end{align*}
where $\sigma_i$ is the $i$-th largest singular value of $\bm M$, $\bm u_i$ and $\bm v_i$ are the corresponding left and right singular vectors.
Since each set of the left and right singular vectors are orthonormal, the degrees of freedom of matrix completion is given by
\begin{align*}
    r + \sum_{i=1}^r (n_1-i) + \sum_{j=1}^r (n_2-j) = (n_1+n_2-r)r ,
\end{align*}
which is significantly less than the total number of entries in $\bm M$ when $r$ is small. This implies the possibility of recovering the entire matrix even when only a few entries are observed. However, not every matrix with more than $(n_1+n_2-r)r$ observed entries can be completed. For instance, if an entire column (or row) of a rank-one matrix is missing, then the matrix cannot be recovered. Similarly, if a low-rank matrix contains too many zero entries, then the observed entries might end up being all zero, thereby not providing any clue about the missing entries. 
The aforementioned argument motivates the two standard assumptions in matrix completion: the incoherence condition and the random sampling model.
Under these assumptions, Cand\`{e}s and Recht \cite{candes2009exact} showed that matrix completion can be solved exactly for most settings of the low-rank matrix $\bm M$ and the sampling set $\Omega$.
This breakthrough has started a long line of research on efficient methods for solving matrix completion.

% \hl{TV: Convex relaxation approach}
In the same work, Cand\`{e}s and Recht \cite{candes2009exact} proposed a convex relaxation approach to matrix completion, replacing the original linearly constrained rank minimization problem by a linearly constrained nuclear norm minimization problem. Their result leads to a well-known class of proximal-type algorithms for nuclear norm minimization \cite{cai2010singular,ma2011fixed,ji2009accelerated,toh2010accelerated} with rigorous mathematical guarantees and extensions of classic acceleration techniques.
Nonetheless, convex-relaxed methods are generally considered slow compared to their non-convex counterparts in practice. On the one hand, interior-point methods for solving the nuclear norm minimization problem are computationally expensive and even infeasible for large matrices.
On the other hand, proximal-type algorithms suffer from slow convergence due to the conservative nature of the soft-thresholding operator \cite{koren2009bellkor,vu2019accelerating}. 

% \hl{TV: IHT-based approach}
Another approach to matrix completion is known as iterative hard thresholding. To address the computational concern from the use of convex relaxation, IHT methods have been proposed to directly solve the non-convex rank minimization problem \cite{jain2010guaranteed,goldfarb2011convergence}. Each IHT iteration takes one step in the opposite direction of the gradient and another step projecting the result onto the set of rank-$r$ matrices.
Since the process resembles hard-thresholding singular values, we refer to the class of algorithms using this technique as iterative hard thresholding.
When the solution is low-rank, hard-thresholding algorithms is more efficient than their soft-thresholding counterparts in both computational complexity per iteration and convergence speed. Variants of plain IHT with faster convergence have also been developed, including normalized IHT \cite{tanner2013normalized}, conjugate gradient IHT \cite{blanchard2015cgiht}, Nesterov's accelerated gradient IHT \cite{vu2019accelerating}, Heavy-Ball IHT \cite{vu2019local}, just to name a few.
The drawback of IHT methods, however, is the lack of mathematical guarantees on their convergence behavior. As pointed out in \cite{jain2010guaranteed}, the restricted isometry property (RIP), which is widely used in establishing the global convergence in matrix sensing, does not hold for matrix completion. Therefore, the global convergence of IHT methods for matrix completion is still an open question. Until recently, the only guarantee on the global convergence of a IHT method, to the best of our knowledge, is provided in \cite{jain2015fast}. In their work, the authors considered a variant of the singular value projection (SVP) algorithm with a resampling scheme and proved the fast linear convergence of the proposed algorithm with a sample complexity that depends on the condition number and desired accuracy. Notwithstanding, this result imposes some limitations at conceptual, practical, and theoretical levels due to the requirement of resampling \cite{sun2016guaranteed}.
In a different perspective, local convergence of IHT methods has also been studied by Chunikhina~\textit{et.~al.} \cite{chunikhina2014performance}. In particular, by considering a special case of the SVP algorithm with unit step size, called iterative hard-thresholded singular value decomposition (IHTSVD), the authors showed that IHTSVD converges linearly to the solution $\bm M$ as long as the algorithm is initialized close enough to $\bm M$. Consequently, this analysis explains the superior performance of IHT methods over proximal-type methods in practice.\footnote{Convergence guarantees on proximal-type methods for matrix completion are often sub-linear \cite{cai2010singular,toh2010accelerated}.}
A similar approach can be found in the unpublished work of Lai and Varghese \cite{lai2017convergence}. However, we remark that while the latter work proves the existence of an upper bound on the linear convergence rate of IHTSVD, the former provides an exact expression of the rate that depends directly on the structure of $\bm M$ and $\Omega$. 

% \hl{TV: Factorization-based approach}
The most popular approach to matrix completion is non-convex factorization. This approach stems from the Burer-Monteiro factorization \cite{burer2005local}, whereby the low-rank matrix is viewed as a product of two low-rank components. The resulting least-squares problem is unconstrained albeit non-convex. Recent progress in this approach has shown that any local minimum of the re-parameterized problem is also a global minimum \cite{sun2016guaranteed,ge2016matrix}. Thus, basic optimization procedures such as gradient descent \cite{chen2015fast,sun2016guaranteed,ma2018implicit} and alternating minimization \cite{chen2012matrix,jain2013low,hardt2014understanding,hardt2014fast} can provably find the global solution at a linear convergence rate. The exact linear convergence rate of gradient descent for matrix completion has recently been studied by Vu and Raich \cite{vu2021exact}.
In Table~\ref{tbl:def}, we summarize the aforementioned approaches to matrix completion and the corresponding algorithms existing in the literature.
\begin{table*}
\begin{center}
\resizebox{\linewidth}{!}{
\begin{tabular}{|>{\raggedright}p{3.2cm}|p{6.2cm}|p{7cm}|}
\hline
\textbf{Problem formulation} & \textbf{Description} & \textbf{Algorithms} \\
\hline
Linearly constrained nuclear norm minimization & $\displaystyle \min_{\bm X \in \R^{n_1 \times n_2}} \norm{\bm X}_* \text{ s.t. } X_{ij}=M_{ij}, \quad (i,j)\in \Omega$ & Semi-definite programming (SDP) \cite{candes2009exact}, singular value thresholding (SVT) \cite{cai2010singular}, accelerated proximal gradient (APG) \cite{toh2010accelerated}, conditional gradient descent (CGD) \cite{jaggi2013revisiting,rao2015forward,boyd2017alternating} \\
\hline
Rank-constrained least squares & $\displaystyle \min_{\bm X \in \R^{n_1 \times n_2}} \sum_{(i,j)\in  \Omega}(X_{ij}-M_{ij})^2 \text{ s.t. } \rank(\bm X) \leq r$ & Singular value projection (SVP) \cite{jain2010guaranteed}, normalized IHT (NIHT) \cite{tanner2013normalized}, conjugate gradient IHT (CGIHT) \cite{blanchard2015cgiht}, iterative hard-thresholded SVD (IHTSVD) \cite{chunikhina2014performance}, accelerated IHT \cite{vu2019accelerating,vu2019local} \\
\hline
Low-rank factorization & $\displaystyle \min_{\bm Y \in \R^{n_1 \times r} , \bm Z \in \R^{n_2 \times r}} \sum_{(i,j)\in \Omega}((\bm Y \bm Z^{\topnew})_{ij}-M_{ij})^2$ & Alternating minimization (AM) \cite{jain2013low,hardt2014understanding}, gradient descent (GD) \cite{sun2016guaranteed,ma2018implicit}, projected gradient descent (PGD) \cite{burer2005local,chen2015fast}, stochastic gradient descent (SGD) \cite{sun2016guaranteed} \\
\hline
\end{tabular}
}
\caption{Three well-known formulations of the matrix completion problem.}
\label{tbl:def}
\end{center}
\end{table*}

% \hl{TV: Our contribution}
This paper is developed based on the work of Chunikhina~\textit{et.~al.} \cite{chunikhina2014performance} on the local convergence of the IHTSVD algorithm for matrix completion. Our main contribution is three-fold. First, we propose a novel analysis of the local convergence of IHTSVD for matrix completion. 
The proposed analysis establishes the region of convergence that is proportional to the least non-zero singular value of $\bm M$. Moreover, we show that the convergence is asymptotically linear and the exact rate can be described in a closed-form expression of the projections onto the (left and right) null spaces of $\bm M$ and the sampling pattern $\Omega$. 
% We compare our result with traditional approaches to matrix completion with global guarantees, which requires additional assumptions (e.g., the incoherence of the solution matrix $\bm M$, the uniformity of the sampling set $\Omega$, the rank $r$ is low, and the sample complexity is sufficiently large) to provide a probabilistic global upper-bound on the convergence rate.
Second, based on the analytical exact linear rate, we utilize random matrix theory to study the asymptotic behavior of IHTSVD in large-scale matrix completion. As the size of $\bm M$ grows to infinity, we uncover the linear rate of IHTSVD converges to a deterministic constant that can be expressed in closed form in terms of the relative rank and the sampling rate. Finally, we present numerical results to verify our proposed exact rate of convergence as well as the asymptotic rate of IHTSVD in large-scale settings.
% \hl{RR: (1) make sure you point out how we distinguish ourselves from other works on linear convergence rate. (2) remember, we want to tell the reader what the paper is about but also that it is novel and offers new insights into the problem.} 
% \hl{TV: closed-form ... spell-out in 3-4 words and develop on it}
% \hl{Go back to the abstract...}

\section{Preliminaries}
\label{sec:prel}

\subsection{Notations}
Throughout the paper, we use the notations $\norm{\cdot}_F$, $\norm{\cdot}_2$, and $\norm{\cdot}_{2,\infty}$ to denote the Frobenius norm, the spectral norm and the $l_2/l_\infty$ norm (i.e., the largest $l_2$ norm of the rows) of a matrix, respectively. Occasionally, $\norm{\cdot}_2$ is used on a vector to denote the Euclidean norm. 
The notation $[n]$ refers to the set $\{ 1,2,\ldots,n \}$. Boldfaced symbols are reserved for vectors and matrices. 
In addition, let $\bm I_n$ denote the $n \times n$ identity matrix. We also use $\otimes$ to denote the Kronecker product between two matrices.

For a matrix $\bm X \in \R^{n_1 \times n_2}$, $X_{ij}$ refers to the $(i,j)$ element of $\bm X$. We denote $\sigma_{\max}(\bm X)$ and $\sigma_{\min}(\bm X)$ as the largest and smallest singular values of $\bm X$, respectively, and denote $\kappa(\bm X) = \sigma_{\max}(\bm X) / \sigma_{\min}(\bm X)$ as the condition number of $\bm X$.
$\vect(\bm X)$ denotes the vectorization of $\bm X$ by stacking its columns on top of one another. 
Let $\bm F(\bm X)$ be a matrix-valued function of $\bm X$. Then, for some $k>0$, we use 
$\bm F(\bm X) = \bm \O(\norm{\bm X}_F^k)$ to imply
\begin{align*}
    \lim_{\delta \to 0} \sup_{\norm{\bm X}_F=\delta} \frac{\norm{\bm F(\bm X)}_F}{\norm{\bm X}_F^k} < \infty .
\end{align*}

\subsection{Background}
Let us use $\bm M$ to denote the underlying $n_1 \times n_2$ real matrix with rank 
\begin{align} \label{equ:n}
    1 \leq r \leq m=\min\{n_1, n_2\}.
\end{align}
The sampling set $\Omega$ is a subset of the Cartesian product $[n_1] \times [n_2]$, with cardinality of $1 \leq s < n_1 n_2$. Furthermore, the orthogonal projection associated with $\Omega$ is given in the following:
\begin{definition}
The orthogonal projection onto the set of matrices supported in $\Omega$ is defined as a linear operator $\P_\Omega: \R^{n_1 \times n_2} \to \R^{n_1 \times n_2}$ satisfying
\begin{align*}
    [\P_\Omega (\bm X)]_{ij} = \begin{cases} X_{ij} &\text{if } (i,j)\in \Omega, \\ 0 &\text{if } (i,j)\in \bar{\Omega} , \end{cases}
\end{align*}
where $\bar{\Omega}$ denotes the complement set of $\Omega$.
\end{definition}
\noindent If we consider vector spaces instead of matrix spaces, the orthogonal projection $\P_\Omega$ can also be viewed as a selection matrix corresponding to $\Omega$:
\begin{definition} \label{def:S}
The selection matrix $\bm S_\Omega \in \R^{n_1 n_2 \times s}$ comprises a subset of $s$ columns of the identity matrix of dimension $n_1n_2$ such that
\begin{align*}
\begin{cases}
    \bm S_\Omega^{\topnew} \bm S_\Omega = \bm I_s , \\
    \vect\bigl(\P_\Omega (\bm X) \bigr) = \bm S_\Omega \bm S_\Omega^{\topnew} \vect(\bm X) .
\end{cases}
\end{align*}
\end{definition}
\noindent Corresponding to the complement set $\bar{\Omega}$, we also define similar notations for $\P_{\bar{\Omega}}: \R^{n_1 \times n_2} \to \R^{n_1 \times n_2}$ and ${\bm S}_{\bar{\Omega}} \in \R^{n_1 n_2 \times (n_1 n_2 - s)}$.

Next, using the notation of $\P_\Omega$, we can formulate the matrix completion problem as follows:
\begin{align} \label{prob:mcp}
    \min_{\bm X \in \R^{n_1 \times n_2}} \frac{1}{2} \norm{\P_{\Omega} (\bm X - \bm M)}_F^2 \text{ s.t. } \rank(\bm X) \leq r .
\end{align}
One natural approach to the optimization problem (\ref{prob:mcp}) is projected gradient descent. Starting at some $\bm X^{(0)}$, we iteratively update the current matrix by (i) taking a step in the opposite direction of the gradient and (ii) projecting the result back onto the set of matrices with rank less than or equal to $r$. It follows that
\begin{align} \label{equ:PGD}
    \bm X^{(k+1)}= \P_r \bigl(\bm X^{(k)} - \eta \P_\Omega (\bm X^{(k)} - \bm M) \bigr) ,
\end{align}
where $\eta$ is the step size and $\P_r$ is the rank-$r$ projection (formally defined later in Definition~\ref{def:Pr}). 
% \hl{RR: it would be worthwhile to discuss the 2 fixed values used for $\eta$ in the literature, i.e., $n_1n_2/s$ (leading to SVP) and $1$ (leading to IHT). Your organization follows this already just need to mention the different choice of $\eta$} 
In the literature, PGD with step size $\eta = n_1n_2/s$ is also known as the Singular Value Projection (SVP) algorithm for matrix completion \cite{jain2010guaranteed}. It is interesting to note that under certain assumptions, \cite{jain2015fast} showed that the algorithm enjoys a fast global linear convergence with this choice of step size.
On the other hand, setting the step size $\eta=1$ yields the following update
\begin{align*}
    \bm X^{(k+1)} &= \P_r \bigl(\bm X^{(k)} - \P_\Omega (\bm X^{(k)} - \bm M) \bigr) \\
    &= \P_r\bigl( \P_{\bar{\Omega}} ( \bm X^{(k)} ) + \P_{\Omega}(\bm M) \bigr) .
\end{align*}
This motivates the IHTSVD algorithm \cite{chunikhina2014performance} that alternates between two projection steps: the projection onto the manifold of rank-$r$ matrices and the projection onto the set of matrices supported in $\Omega$ (see Algorithm~\ref{algo:IHTSVD}). This paper, developed based on \cite{chunikhina2014performance}, focuses on local convergence properties of IHTSVD. Compared to the existing global convergence analysis for matrix completion, our setting does not require certain assumptions such as the incoherence of $\bm M$, the uniform randomness of $\Omega$, and the low sample complexity, e.g., $s=\O(r^5 n \log n)$ in \cite{jain2015fast}. We also note that the proposed analysis can be extended to other variants of PGD with different step sizes.
% \hl{RR: do we want to say that our analysis can be used to also analyze other variants of SVP?}

Finally, we present a formal definition of the rank-$r$ projection.
Consider a matrix $\bm X \in \R^{n_1 \times n_2}$ with the singular value decomposition
\begin{align*}
    \bm X = \sum_{i=1}^{m} \sigma_i(\bm X) \bm u_i(\bm X) \bm v_i^{\topnew}(\bm X) ,  
\end{align*}
where $\sigma_1(\bm X) \geq \ldots \geq \sigma_{m}(\bm X) \geq 0$ are the singular values of $\bm X$ and $\{\bm u_1(\bm X),\ldots, \bm u_{m}(\bm X)\}$, $\{\bm v_1(\bm X),\ldots, \bm v_{m}(\bm X)\}$ are the sets of left and right singular vectors of $\bm X$, respectively.
\begin{definition} \label{def:Pr}
The rank-$r$ projection of $\bm X$ is defined as
\begin{align*}
    \P_r(\bm X) = \sum_{i=1}^r \sigma_{i}(\bm X) \bm u_{i}(\bm X) \bm v_{i}^{\topnew}(\bm X) .
\end{align*}
\end{definition}
\noindent The rank-$r$ projection of $\bm X$ is unique if and only if $\sigma_r(\bm X) > \sigma_r(\bm X)$ or $\sigma_r(\bm X) = 0$ \cite{eckart1936approximation}. Since $\P_r(\bm X)$ zeroes out all the small singular value of $\bm X$, it is often referred as the singular value hard-thresholding operator.
Since $\bm M$ is a rank-$r$ matrix, we have
\begin{align*}
    \bm M = \P_r(\bm M) = \sum_{i=1}^r \sigma_i \bm u_i \bm v_i^{\topnew} = \bm U_r \bm \Sigma_r \bm V_r^\topnew ,
\end{align*}
where $\bm \Sigma_r=\diag(\sigma_1,\ldots,\sigma_r)$ contains the singular values of $\bm M$ and $\bm U_r = [\bm u_1,\ldots, \bm u_r] \in \R^{n_1 \times r}$, $\bm V_r = [\bm v_1,\ldots,\bm v_r] \in \R^{n_2 \times r}$ are comprised of the first $r$ left and right singular vectors of $\bm M$, respectively.\footnote{In the rest of this paper, we omit the parameter in the notation of the singular values and the singular vectors of $\bm M$ for simplicity.}
Denote $\bm U_{\perp} = [\bm u_{r+1},\ldots, \bm u_{n_1}] \in \R^{n_1 \times (n_1-r)}$ and $\bm V_{\perp} = [\bm v_{r+1},\ldots,\bm v_{n_2}] \in \R^{n_2 \times (n_2-r)}$. The projections onto the left and right null spaces of $\bm M$ are uniquely defined as $\bm P_{\bm U_{\perp}} = \bm U_{\perp} \bm U_{\perp}^{\topnew} = \bm I_{n_1} - \sum_{i=1}^r \bm u_i \bm u_i^{\topnew}$ and $\bm P_{\bm V_{\perp}} = \bm V_{\perp} \bm V_{\perp}^{\topnew} = \bm I_{n_2} - \sum_{i=1}^r \bm v_i \bm v_i^{\topnew}$, respectively.

\subsection{Related Work}
Traditional approaches to matrix completion often make assumptions on the incoherence of the underlying matrix $\bm M$ and the randomness of the sampling set. First, the incoherence condition for matrix completion, introduced by Cand\`{e}s and Recht \cite{candes2009exact}, is stated as:  
\begin{assumption}[Incoherence]
The matrix $\bm M = \bm U_r \bm \Sigma_r \bm V_r^\topnew$ is $\mu$-incoherent, i.e.,
\begin{align*}
    \norm{\bm U_r}_{2,\infty} \leq \sqrt{\frac{\mu r}{n_1}} \text{ and } \norm{\bm V_r}_{2,\infty} \leq \sqrt{\frac{\mu r}{n_2}} .
\end{align*}
\end{assumption}
\noindent Intuitively, an incoherent matrix has well-spread singular vectors and is less likely in the null space of the sampling operator. A common setting that generates incoherent matrices is the random orthogonal model:
\begin{definition}[Random orthogonal model]
The Haar measure provides a uniform and translation-invariant distribution over the group of orthogonal matrices $\mathbb{O}(n)$. $\bm M$ is said to follow a random orthogonal model if $\bm U_r$ and $\bm V_r$ are sub-matrices of Haar-distributed matrices in $\mathbb{O}({n_1})$ and $\mathbb{O}({n_2})$, respectively.
\end{definition}
\noindent Second, to avoid adversarial patterns in the sampling set, it is common to assume that each entry in $\Omega$ is selected randomly:
\begin{assumption}[Uniform sampling] \label{assp:omega}
The sampling set $\Omega$ is obtained by selecting $s$ elements uniformly at random from the Cartesian product $[n_1] \times [n_2]$.
\end{assumption}
\noindent We note that a similar but not equivalent assumption on the sampling set is the Bernoulli model in which each entry of $\bm M$ is observed independently with probability $s/n_1n_2$ \cite{sun2016guaranteed}. Under these two standard assumptions, Cand\`{e}s and Recht \cite{candes2009exact} showed that symmetric matrix completion of size $n$ can be solved exactly provided that the number of observations is sufficiently large, i.e., $s=\O(n^{1.2}r \log n)$. Later on, global convergence guarantees for various matrix-completion algorithms have been actively developed, with improved bounds on the sample complexity. Examples of these works include \cite{recht2011simpler, hardt2014fast,jain2015fast,sun2016guaranteed,ma2018implicit}. It is worthwhile mentioning that ideally, one would like to recover the low-rank matrix from a minimum number of observations, which is in the order of the degrees of freedom of the problem, i.e., $\O(nr)$.

In this paper, we study the convergence of IHT for matrix completion from a different perspective. Without any assumptions about the incoherence of $\bm M$ and the randomness of the sampling set $\Omega$, we identify a deterministic condition on the structure of $\bm M$ and $\Omega$ such that the local linear convergence of IHTSVD can be guaranteed. Compared to the aforementioned bounds on the global convergence rate, our result is exact and tighter thanks to the exploitation of the local structure of the problem.
Our technique utilizes the recently developed error bound for the first-order Taylor expansion of the rank-$r$ projection, proposed by Vu~\textit{et.~al.} in \cite{vu2021perturbation}. The result is rephrased below.
\begin{proposition}[Rephrased from \cite{vu2021perturbation}]
\label{prop:rankop}
For any $\bm \Delta \in \R^{n_1 \times n_2}$, we have
\begin{align} \label{equ:rankop}
\P_r (\bm M + \bm \Delta) = \bm M + \bm \Delta - \bm P_{\bm U_{\perp}} \bm \Delta \bm P_{\bm V_{\perp}} + \bm R(\bm \Delta) ,
\end{align}
where the residual $\bm R: \R^{n_1 \times n_2} \to \R^{n_1 \times n_2}$ satisfies:
\begin{align*}
    \norm{\bm R(\bm \Delta)}_F \leq  \frac{c_1}{\sigma_r} \norm{\bm \Delta}_F^2 ,
\end{align*}
for some universal constant $1 + 1/\sqrt{2} \leq c_1 \leq 4(1+\sqrt{2})$. 
\end{proposition}
The rest of the paper is organized as follows. In Section~\ref{sec:local}, we provide the local convergence analysis of IHTSVD for matrix completion and the proof of the main result. Next, Section~\ref{sec:asymptotic} presents a summary of related results in random matrix theory, followed by our novel result on the asymptotic behavior of the convergence rate in large-scale settings. The numerical results to verify the analysis in Sections~\ref{sec:local} and \ref{sec:asymptotic} are given in Section~\ref{sec:exp}. Finally, we put the detailed proofs of all the main theorems and lemmas in the appendix.

\section{Local Convergence of IHTSVD}
\label{sec:local}

This section presents our analysis of local convergence of IHTSVD. First, we leverage the results in perturbation analysis to identify the Taylor series expansion of the rank-$r$ projection. Next, the approximation allows us to derive the nonlinear difference equation that describes the change in the distance to the local optimum through IHT iterations. Closed-form expressions of the asymptotic convergence rate and the region of convergence are also given as a result of our analysis. 

\begin{algorithm}[t]
\caption{IHTSVD}
\label{algo:IHTSVD}
\begin{algorithmic}[1]
\Require{$\P_{\Omega}(\bm M)$, $r$, $K$, $\bm X^{(0)}$}
\Ensure{$\bm X^{(K)}$}
\For{$k=0,1,\ldots, K-1$}
\State $\bm X^{(k+1)}=\P_r\bigl( \P_{\bar{\Omega}} ( \bm X^{(k)} ) + \P_{\Omega}(\bm M) \bigr)$
\EndFor
\end{algorithmic}
\end{algorithm}

\subsection{Main Result}
Our local convergence result is stated as follows:
\begin{theorem}
\label{theo:IHT}
Let $\{\bm X^{(k)}\}_{k=0}^\infty$ be the sequence of matrices generated by Algorithm~\ref{algo:IHTSVD}, i.e., 
\begin{align} \label{equ:update}
    \bm X^{(k+1)}=\P_{\bar{\Omega}} \bigl( \P_r ( \bm X^{(k)} ) \bigr) + \P_{\Omega}(\bm M) 
\end{align}
for all integer $k$, and $\bm X^{(0)}$ satisfies
\begin{align} \label{equ:ROC}
    \norm{\bm X^{(0)} - \bm M}_F < \frac{\lambda_{\min}(\bm H)}{c_1} \sigma_r ,
\end{align}
where $\bm H$ is an $(n_1 n_2 - s)$ square matrix given by
\begin{align} \label{equ:H}
    \bm H = {\bm S}_{\bar{\Omega}}^{\topnew} (\bm P_{\bm V_{\perp}} \otimes \bm P_{\bm U_{\perp}}) {\bm S}_{\bar{\Omega}} .
\end{align}
Then, $\norm{\bm X^{(k)} - \bm M}_F$ converge asymptotically at a linear rate
\begin{align} \label{equ:rho}
    \rho = 1 - \lambda_{\min}(\bm H) .
\end{align}
Specifically, for any $\epsilon>0$, $\norm{\bm X^{(k)} - \bm M}_F \leq \epsilon \norm{\bm X^{(0)} - \bm M}_F$ for all integer $k$ such that 
\begin{align} \label{equ:Ne}
    k \geq K(\epsilon) = \frac{\log(1/\epsilon)}{\log(1/(1-\lambda_{\min}(\bm H)))} + c ,
\end{align}
where $\tau = \frac{c_1 \norm{\bm X^{(0)} - \bm M}_F}{\sigma_r \lambda_{\min}(\bm H)}$ and
\begin{align*}
    c = &\frac{1}{\rho \log(1/\rho)} \Biggl( E_1\Bigl(\log\frac{1}{\rho+\tau(1-\rho)}\Bigr) - E_1\Bigl(\log\frac{1}{\rho}\Bigr) \\
    &+ \frac{1}{2} \cdot \log \biggl( \frac{\log(1\rho/\rho)}{\log \bigl(1/(\rho+\tau(1-\rho))\bigr)} \biggr) \Biggr) + 1 , \numberthis \label{equ:c3}
\end{align*}
with $E_1(t) = \int_t^\infty \frac{e^{-z}}{z}dz$ being the exponential integral \cite{milton1964handbook}.
\end{theorem}
Theorem~\ref{theo:IHT} provides a closed-form expression of the linear convergence rate of IHTSVD for matrix completion. 
As can be seen in (\ref{equ:Ne}), the speed of convergence depends strongly on how close the smallest eigenvalue of $\bm H$ is to zero: as $\lambda_{\min}(\bm H)$ approaches $0$, the number of iterations needed to reach a relative accuracy of $\epsilon$, i.e., $K(\epsilon)$, grows to infinity. 
When $\lambda_{\min}(\bm H)=0$, the condition in (\ref{equ:ROC}) cannot be satisfied and hence, there is no linear convergence guarantee provided by our theorem in this case.
On the other hand, from (\ref{equ:H}), one can verify that all eigenvalues of $\bm H$ lie between $0$ and $1$ since the norm of either a projection matrix or a selection matrix is less than or equal to $1$. This combined with the aforementioned condition that $\lambda_{\min}(\bm H)>0$ ensures the linear convergence rate $\rho$ in (\ref{equ:rho}) belongs to $[0,1)$.
\begin{remark} \label{rmk:sr}
Theorem~\ref{theo:IHT} does not guarantee linear convergence when $\lambda_{\min}(\bm H)=0$. Interestingly, one such situation is when $\bm H$ is \textbf{rank-deficient}. Let us represent
\begin{align*}
    \bm H &= {\bm S}_{\bar{\Omega}}^{\topnew} (\bm V_{\perp} \otimes \bm U_{\perp}) (\bm V_{\perp} \otimes \bm U_{\perp})^{\topnew} {\bm S}_{\bar{\Omega}} \\
    &= \bm W \bm W^{\topnew} ,
\end{align*}
where $\bm W = {\bm S}_{\bar{\Omega}}^{\topnew} (\bm V_{\perp} \otimes \bm U_{\perp}) \in \R^{(n_1 n_2 - s) \times (n_1-r)(n_2-r)}$.
If $\bm W$ is a tall matrix, i.e., 
\begin{align} \label{equ:s_cond}
    s < (n_1+n_2-r)r ,
\end{align}
then it follows that $\bm H$ is rank-deficient  and $\lambda_{\min}(\bm H) = 0$. We note that in this case the number of sampled entries is less than the degrees of freedom of the problem. 
\end{remark}

\begin{remark}
When $s \geq (n_1+n_2-r)r$, it is possible that $\lambda_{\min}(\bm H) = 0$ for certain (adversarial) sampling patterns. For example, consider a $3 \times 2$ rank-$1$ matrix 
\begin{align*}
    \bm M = \begin{bmatrix} 1 & 0 \\ 0 & 0 \\ 0 & 0 \end{bmatrix} = \begin{bmatrix} 1 \\ 0 \\ 0 \end{bmatrix} \cdot \begin{bmatrix} 1 & 0 \end{bmatrix}^{\topnew} .
\end{align*}
One choice of the matrices $\bm U_\perp$ and $\bm V_\perp$ is
\begin{align*}
    \bm U_\perp = \begin{bmatrix} 0 & 0 \\ 1 & 0 \\ 0 & 1 \end{bmatrix} \text{ and } \bm V_\perp = \begin{bmatrix} 0 \\ 1 \end{bmatrix} .
\end{align*}
If we observe $s=4$ entries of the first two rows of $\bm M$, namely, $(1,1)$, $(1,2)$, $(2,1)$, and $(2,3)$, the selection matrix corresponding to the unobserved entries $(3,1)$ and $(3,2)$ is given by
\begin{align*}
     \bm S_{\bar{\Omega}}^{\topnew} = \begin{bmatrix} 0 & 0 & 1 & 0 & 0 & 0 \\ 0 & 0 & 0 & 0 & 0 & 1 \end{bmatrix} .
\end{align*}
Then, we have
\begin{align*}
    \bm H = {\bm S}_{\bar{\Omega}}^{\topnew} (\bm V_{\perp} \otimes \bm U_{\perp}) (\bm V_{\perp} \otimes \bm U_{\perp})^{\topnew} {\bm S}_{\bar{\Omega}} = \begin{bmatrix} 0 & 0 \\ 0 & 1 \end{bmatrix} 
\end{align*}
and $\lambda_{\min}(\bm H) = 0$.
While Theorem~\ref{theo:IHT} does not guarantee linear convergence of IHTSVD, one may realize that it is impossible to recover the last row of $\bm M$ in this case.
\end{remark}

\subsection{Proof of Theorem~\ref{theo:IHT}}

% \hl{RR: proof sketch is often more appropriate in conference papers. I think that you do have the full proof except that details are deferred to the appendix. With a proof sketch details are often omitted, but I think we do not omit details.}
This section provides the proof of Theorem~\ref{theo:IHT}.
We starts by formulating the recursion on the error matrix from the update (\ref{equ:update}) and the linearization of the rank-$r$ projection:
\begin{lemma} \label{lem:error_recur}
Let us define the error matrix and its economy vectorized version, respectively, as
\begin{align*}
    \bm E^{(k)} = \bm X^{(k)} - \bm M \qquad \text{and} \qquad \bm e^{(k)} = {\bm S}_{\bar{\Omega}}^{\topnew} \vect(\bm E^{(k)}) .
\end{align*}
Then, we have
\begin{align} \label{equ:E}
    \bm E^{(k+1)} &= \P_{\bar{\Omega}} \bigl( \bm E^{(k)} - \bm P_{\bm U_{\perp}} \bm E^{(k)} \bm P_{\bm V_{\perp}} + \bm R(\bm E^{(k)}) \bigr) 
\end{align}
and
\begin{align} \label{equ:e}
    \bm e^{(k+1)} &= \bigl(\bm I - {\bm S}_{\bar{\Omega}}^{\topnew} (\bm P_{\bm V_{\perp}} \otimes \bm P_{\bm U_{\perp}}) {\bm S}_{\bar{\Omega}} \bigr) \bm e^{(k)} + \bm r \bigl( \bm e^{(k)} \bigr) , 
\end{align}
where $\bm R(\cdot)$ is the residual defined in Proposition~\ref{prop:rankop} and $$\bm r (\bm e) = {\bm S}_{\bar{\Omega}}^{\topnew} \vect \Bigl(\bm R \bigl({\vect}^{-1}({\bm S}_{\bar{\Omega}} \bm e) \bigr) \Bigr) \quad \text{ for } \bm e \in \R^{n_1 n_2-s} .$$
Here we recall that $\vect^{-1}(\cdot)$ is the inverse vectorization operator such that $(\vect^{-1} \circ \vect)$ is identity.
\end{lemma}
\noindent Note that $\bm E^{(k)}$ belongs to the set of matrices supported in $\Omega$ and hence, $\norm{\bm E^{(k)}}_F = \norm{\bm e^{(k)}}_2$.
Next, using the definition of the operator norm, one can obtain the following bound on the norm of the error matrix:
\begin{lemma} \label{lem:error_ineq}
The Frobenius norm of the error matrix satisfies
\begin{align} \label{equ:Einq}
    \norm{\bm E^{(k+1)}}_F \leq \bigl(1-\lambda_{\min}(\bm H)\bigr) \norm{\bm E^{(k)}}_F + \frac{c_1}{\sigma_r} \norm{\bm E^{(k)}}_F^2 .
\end{align}
\end{lemma}
\noindent The nonlinear difference equation (\ref{equ:Einq}) has been well-studied in the stability theory of difference equations \cite{bellman2008stability,polyak1964some,vu2021closed}. In fact, our theorem follows directly on applying Theorem~1 in \cite{vu2021closed} to (\ref{equ:Einq}), with $a_0 = \norm{\bm E^{(0)}}_F$, $\rho=1-\lambda_{\min}(\bm H)$, and $q={c_1}/{\sigma_r}$.
The proofs of Lemmas~\ref{lem:error_recur} and \ref{lem:error_ineq} are given in Appendix~\ref{appdx:IHT}.

\section{Convergence of IHTSVD for Large-Scale Matrix Completion}
\label{sec:asymptotic}

In this section, we study the convergence of IHTSVD for large-scale matrix completion, a setting of practical interest in the rise of big data. 
Using recent results in random matrix theory, we show that, as its dimensions grow to infinity, the spectral distribution of $\bm H$ converges almost surely to a deterministic distribution with a bounded support. 
Consequently, we propose a large-scale asymptotic estimate of the linear convergence rate of IHTSVD that is a closed-form expression of the relative rank and the sampling rate. 

\subsection{Overview}
We are interested in the asymptotic setting in which the size of $\bm M$ grows to infinity, i.e., $m = \min\{n_1,n_2\} \to \infty$. Let us assume that the ratio $n_1/n_2$ remains to be a non-zero constant as $m \to \infty$.
In addition, we introduce two concepts that are the normalization of the degrees of freedom and the number of measurements:
\begin{definition}[Relative rank] \label{assp:rho_r}
The rank $r$ increases as $m \to \infty$ such that the relative rank remains to be a constant
\begin{align} \label{equ:rho_r}
    \rho_r = 1-\sqrt{\bigl( 1 - \frac{r}{n_1} \bigr) \bigl( 1 - \frac{r}{n_2} \bigr)} \in (0,1] .
\end{align}
\end{definition}
\begin{definition}[Sampling rate] \label{assp:rho_s}
The number of observations increases as $m \to \infty$ such that the sampling rate remains to be a constant
\begin{align} \label{equ:rho_s}
    \rho_s = \frac{s}{n_1 n_2} \in (0,1] .
\end{align}
\end{definition}
\noindent When $\rho_s < 1-(1-\rho_r)^2$, we recover the case in Remark~\ref{rmk:sr} where the number of measurements is less than the degrees of freedom. As far as the local linear rate of IHTSVD is concerned, we only consider the case $\rho_s \geq 1-(1-\rho_r)^2$.
\begin{remark}
When $r=m$, we have $\rho_r = 1$. Moreover, when $n_1=n_2=m$, the relative rank is exactly the ratio $r/m$. As can be seen below, the proposed definition of the relative rank incorporates both dimensions of $\bm M$ to enable the compact representation of $\rho$ in terms of $\rho_r$ and $\rho_s$. 
\end{remark}
\noindent We are in position to state our result on the asymptotic behavior of the linear rate $\rho$ in large-scale matrix completion:
\begin{theorem}[Informal] \label{theo:rho_asymp_informal}
For $\rho_s > 1-(1-\rho_r)^2$, the linear convergence rate $\rho$ of IHTSVD approaches 
\begin{align} \label{equ:p_infty}
    \rho_\infty = 1 - \Bigl( \sqrt{(1-\rho_r)^2 \rho_s} - \sqrt{\rho_r(2-\rho_r)(1-\rho_s)} \Bigr)^2 ,
\end{align}
as $m \to \infty$.
\end{theorem}
\noindent The formal statement of our result is given later in Theorem~\ref{theo:rho_asymp}.
Note that $\rho_\infty$ is independent of the structure of the solution matrix $\bm M$ and the sampling set $\Omega$. Moreover, it depends only on the relative rank and the sampling rate. Figure~\ref{fig:asymptotic101} depicts the contour plot of $\rho_\infty$ as a function of $\rho_r$ and $\rho_s$. It can be seen that for a fixed value of $\rho_r$, the asymptotic rate decreases towards $0$ as the number of observed entries increases. This matches with the intuition that more information leads to faster convergence. Conversely, for a fixed value of $\rho_s$, the algorithm converges slower as the rank of the matrix increases, due to the increasing uncertainty (i.e., more degrees of freedom) in the set $\bar{\Omega}$.
On the boundary where $\rho_s = 1 - (1-\rho_r)^2$, there is no linear convergence predicted by our theory since $\rho_\infty = 1$. In this case, we recall that the number of observed entries equals the degrees of freedom of the problem.
\begin{figure}
    \centering
    \includegraphics[scale=.6]{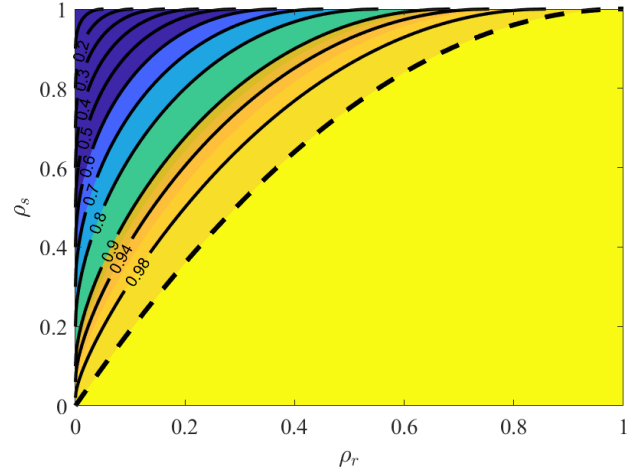}
    \caption{Contour plot of $\rho_\infty$ as a 2-D function of $\rho_r$ and $\rho_s$ given by (\ref{equ:p_infty}). The isoline at which $\rho_\infty=1$ is represented by the dashed line. The yellow region below this isoline corresponds to the under-determined setting $\rho_s<1-(1-\rho_r)$. }
    \label{fig:asymptotic101}
\end{figure}

Our technique relies on recent results in random matrix theory to exploit the special structure of $\bm H$. First, when $n_1/n_2$ remains constant, it holds that $n=n_1n_2 \to \infty$ as $m \to \infty$.
Then, $\bm H$ can be viewed as an element of a sequence of matrices of form
\begin{align} \label{equ:Hn}
    \bm H_n = \bm W_{pq}^{n} (\bm W_{pq}^{n})^{\topnew} ,
\end{align}
where $\bm W_{pq}^{n} \in \R^{pn_1n_2 \times qn_1n_2}$ is a truncation of the orthogonal matrix $\bm W^{n} = \bm V^{n_2} \otimes \bm U^{n_1}$, for $\bm U^{n_1}$ and $\bm V_{\perp}^{n_2}$  orthogonal matrices of dimensions $n_1 \times n_1$ and $n_2 \times n_2$, respectively, and
\begin{align*}
    p &= \frac{n_1 n_2 - s}{n_1n_2} = 1-\rho_s , \\ 
    q &= \frac{(n_1-r)(n_2-r)}{n_1n_2} = (1-\rho_r)^2 .
\end{align*}
As $n$ grows to infinity, we are interested in finding the limit (or even the limiting distribution) of the smallest eigenvalue of $\bm H_n$, which is a random truncation of the Kronecker product of two large dimensional semi-orthogonal matrices.

\subsection{Truncations of Large Dimensional Orthogonal Matrices}

\begin{figure*}
    \centering
    \begin{subfigure}[b]{0.49\textwidth}
        \centering
        \includegraphics[width=\textwidth]{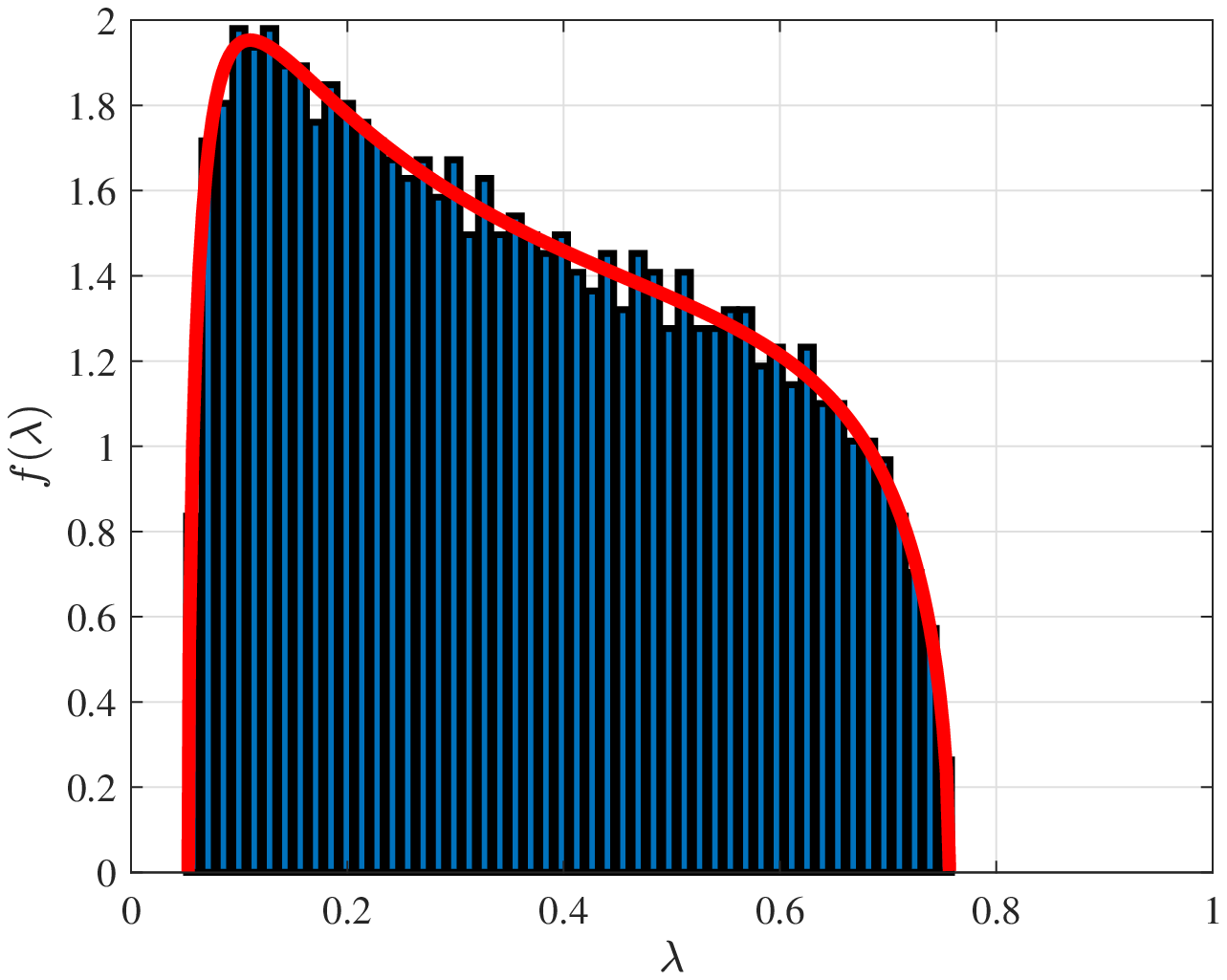}
        \caption{}
    \end{subfigure}
    \begin{subfigure}[b]{0.49\textwidth}
        \centering
        \includegraphics[width=\textwidth]{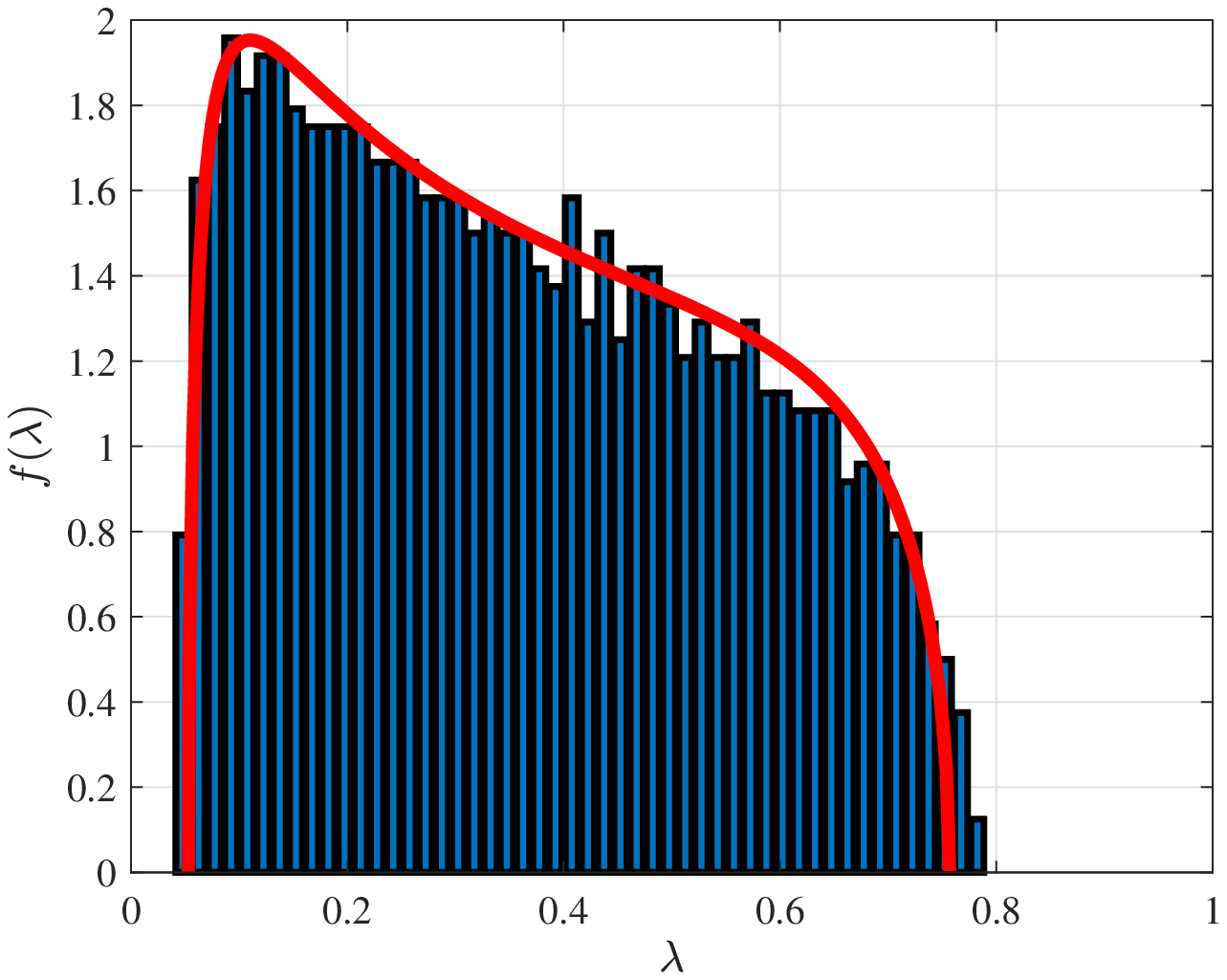}
        \caption{}
    \end{subfigure}
    \caption{Scaled histogram and the limiting ESD of $\bm H_n = \bm W_{pq}^{n} (\bm W_{pq}^{n})^{\topnew}$, where $\bm W_{pq}^{n}$ is the $pn \times qn$ upper-left corner of an $n \times n$ orthogonal matrix $\bm W_n$, for $n=10000$, $p=0.16$, and $q=0.36$. In (a), $\bm W_n$ is the orthogonal factor in the QR factorization of a $10000 \times 10000$ random matrix with $i.i.d$ standard normal entries. In (b), $\bm W_n = \bm Q_1 \otimes \bm Q_2$, where $\bm Q_1$ and $\bm Q_2$ are the orthogonal factors in the QR factorization of two independent $100 \times 100$ random matrices with $i.i.d$ standard normal entries. The histograms with $50$ bins (blue) are scaled by a factor of $1/pnw$, where $w$ is the bin width. The limiting ESD (red) is generated by (\ref{equ:ESD}). It can be seen that the histogram in (a) match the limiting ESD better than the histogram in (b).}
    \label{fig:ESD}
\end{figure*}

Random matrix theory studies the asymptotic behavior of eigenvalues of matrices with entries drawn randomly from various matrix ensembles such as Gaussian orthogonal ensemble (GOE), Wishart ensemble, MANOVA ensemble \cite{edelman2005random}. The closest random matrix ensemble to our matrix ensemble $\{\bm H_n\}_{n \in {\mathbb N}^+}$ is the MANOVA ensemble in which truncations of large dimensional Haar orthogonal matrices are considered. Here we recall that the Haar measure provides a uniform distribution over the set of all $n \times n$ orthogonal matrices $\mathbb{O}(n)$. Indeed, it is a unique translation-invariant probability measure on $\mathbb{O}(n)$.
If we assume that the matrix $\bm M$ follows a random orthogonal model \cite{candes2009exact}, then $\bm U_{\perp}$ and $\bm V_{\perp}$ are essentially sub-matrices of Haar orthogonal matrices in $\mathbb{O}(n_1)$ and $\mathbb{O}(n_2)$, respectively, and $\{\bm H_n\}_{n \in {\mathbb N}^+}$ is a sequence of truncations of the Kronecker product of two Haar orthogonal matrices.

There have been certain theoretical works on truncations of Haar invariant matrices in the literature. In 1980, Wachter \cite{wachter1980limiting} established the limiting distribution of the eigenvalues in the MANOVA ensemble. Later on, the density function of the eigenvalues of such matrix has been shown to be the same as that of a Jacobi matrix \cite{capitaine2004asymptotic,collins2005product,forrester2006quantum}. Shortly afterward, Johnstone proved the Tracy-Widom behavior of the largest eigenvalue in \cite{johnstone2008multivariate}. More recently, Farrell and Nadakuditi relaxed the constraint on the uniform (Haar) distribution of the orthogonal matrix considered the Kronecker products of Haar-distributed orthogonal matrices, which is similar to our matrix completion setting in this paper. The authors showed that the limiting density of their truncations remains the same as the original case without Kronecker products. Further results on the eigenvalue distribution of truncations of Haar orthogonal matrices were also given in \cite{zyczkowski2000truncations,jiang2009approximation,dong2012circular}.
To the best of our knowledge, no result has been shown for the limiting behavior of the smallest eigenvalue of random MANOVA matrices.

In our context, we leverage the recent result in \cite{raich2016eigenvalue}, which assumes the randomness on the truncation rather than the orthogonal matrix. This variant, while differs from the classic MANOVA ensemble in random matrix theory, is well-suited to the setting of matrix completion. Let us begin with the following definition of the empirical spectral distribution:
\begin{definition}
Let $\bm H_n$ be an $n \times n$ real symmetric matrix with eigenvalues $\lambda_1,\ldots,\lambda_n$. The \textbf{empirical spectral distribution (ESD)} of $\bm H_n$, denoted by $\mu_{\bm H_n}$, is the probability measure which puts equal mass at each of the eigenvalues of $\bm H_n$:
\begin{align*}
    \mu_{\bm H_n} \triangleq \frac{1}{n} \sum_{i=1}^n \delta_{\lambda_i} ,
\end{align*}
where $\delta_\lambda$ is the Dirac mass at $\lambda$.
\end{definition}
\noindent Next, we define the concepts of a sequence of row sub-sampled matrices and the concentration property:
\begin{definition} \label{def:row}
For each $n \in {\mathbb N}^+$, consider the $n \times qn$ matrix $\bm W_{q}^{n} = [\bm w_1^n,\ldots, \bm w_n^n]^{\topnew}$, where $\bm w_i^n \in \R^{qn}$ and $q$ is a constant in $(0,1)$.
Let $P_n$ be a $pn$-permutation of $[n]$ selected uniformly at random, for $p$ is a constant in $(0,1)$, and $\bm W_{pq}^{n} \in \R^{pn \times qn}$ be the random matrix obtained by selecting the corresponding set of $pn$ rows from $\bm W_{q}^{n}$. Then, the sequence $\{ \bm W_{q}^{n} \}_{n \in {\mathbb N}^+}$ is called \textbf{a sequence of $q$-tall matrices}, and the sequence $\{ \bm W_{pq}^{n} \}_{n \in {\mathbb N}^+}$ is called \textbf{a sequence of row sub-sampled matrices} of $\{ \bm W_{q}^{n} \}_{n \in {\mathbb N}^+}$.
\end{definition}

\begin{definition} \label{def:concentrated}
% (see (8) in \cite{farrell2013local})
Given the setting in Definition~\ref{def:row}, for each $j \in P_n$, denote $P_n^j = P_n \setminus \{j\}$. In addition, for $z \in {\mathbb C}$, define 
\begin{align*}
    \bm R_j(z) = \Bigl(\sum_{i \in P_n^j} \bm w_i^n (\bm w_i^n)^{\topnew} - z \bm I_{qn}\Bigr)^{-1} . 
\end{align*}
Then, the sequence $\{ \bm W_{q}^{n} \}_{n \in {\mathbb N}^+}$ is \textbf{concentrated} if and only if for any $j \in P_n$ and $z \in {\mathbb C}$, we have
\begin{align} \label{equ:concentrated}
    (\bm w_j^n)^{\topnew} \bm R_j(z) \bm w_j^n - {\mathbb E}_{j \mid P_n^j} \bigl[ (\bm w_j^n)^{\topnew} \bm R_j(z) \bm w_j^n \bigr] \overset{\text{p}}{\to} 0 .
\end{align}
\end{definition}
\noindent In the following, we consider examples of sequences of matrices that are concentrated, as well as an example of the sequence of incoherent matrices that are \textbf{not} concentrated.
\begin{example} \label{eg:concentrated}
Random settings:\footnote{The detail of this example is provided in the Supplementary Material.}
\begin{enumerate}
    \item The sequence of $q$-tall matrices $\{ \bm A_{q}^{n} \}_{n \in {\mathbb N}^+}$, where the entries of $\bm A_{q}^{n}$ are $i.i.d$ $\mathcal{N}(0,1/n)$, is concentrated.
    \item The sequence $\{ \bm B_{q}^{n} \otimes \bm C_{q}^{n} \}_{n \in {\mathbb N}^+}$, where $\{ \bm B_{q}^{n} \}_{n \in {\mathbb N}^+}$ and $\{ \bm C_{q}^{n} \}_{n \in {\mathbb N}^+}$ are two sequences of $q$-tall matrices whose entries are $i.i.d$ $\mathcal{N}(0,1/n)$, is also concentrated.
\end{enumerate}
\end{example}
\begin{example} \label{eg:non_concentrated}
Deterministic settings:
\begin{enumerate}
    \item The sequence of $q$-tall matrices $\{ \bm D_{q}^{n} \}_{n \in {\mathbb N}^+}$, where the entries of $\bm D_{q}^{n}$ are all $1$, is concentrated. 
    \item The sequence of $1/2$-tall matrices $\{ \bm E_{q}^{n} \}_{n \in {\mathbb N}^+}$ where
    \begin{align*}
        \bm E_q^n = \begin{bmatrix}
            0.6 \sqrt{\frac{2}{n}} \mathtt{\bm H}_{n/2} \\ 0.8 \sqrt{\frac{2}{n}} \mathtt{\bm H}_{n/2}
        \end{bmatrix} ,
    \end{align*}
    for $\mathtt{\bm H}_{n/2}$ being a Hadamard matrix of order $n/2$ \cite{hedayat1978hadamard}, is not concentrated. On the other hand, one can verify that $\bm E_q^n$ is $\mu$-incoherent, for 
    \begin{align*}
        \mu = \norm{0.8 \sqrt{2/n} \mathtt{\bm H}_{n/2}}_F^2 \frac{n}{n/2} = 1.28 .
    \end{align*}
    Thus, the concentration assumption in Definition~\ref{def:concentrated} is stronger than the widely-used incoherence assumption.
\end{enumerate}
\end{example}
With these definitions in place, we now state the result on the limiting ESD of a truncation of orthogonal matrices. To fit our matrix completion setting in this paper, we rephrase the result in \cite{raich2016eigenvalue} to the case of row sub-sampled semi-orthogonal matrices (as opposed to column sub-sampled semi-orthogonal matrices in the aforementioned paper).
\begin{proposition}[Rephrased from \cite{raich2016eigenvalue}]
\label{prop:SSP14}
Let $\{ \bm W_{q}^{n} \}_{n \in {\mathbb N}^+}$ be a sequence of $q$-tall matrices that is concentrated. In addition, assume that $\bm W_{q}^{n}$ is semi-orthogonal for all $n \in N^+$, i.e., $(\bm W_{q}^{n})^{\topnew} \bm W_{q}^{n} = \bm I_{qn}$. Let $\{ \bm W_{pq}^{n} \}_{n \in {\mathbb N}^+}$ be a sequence of row sub-sampled matrices of $\{ \bm W_{q}^{n} \}_{n \in {\mathbb N}^+}$.
Then, as $n \to \infty$, the ESD of $\bm H_n = \bm W_{pq}^{n} (\bm W_{pq}^{n})^{\topnew}$ converges almost surely to the deterministic distribution $\mu_{pq}$ such that
\begin{align*}
d\mu_{pq} &= \Bigl(1-\frac{q}{p}\Bigr)_+ \delta(x) dx + \Bigl(\frac{p+q-1}{p}\Bigr)_+ \delta(x-1) dx \\
&\qquad + \frac{\sqrt{(\lambda^+-x)(x-\lambda^-)}}{2\pi px(1-x)} \mathbb{I}[\lambda^- \leq x \leq \lambda^+] dx, \numberthis \label{equ:ESD}
\end{align*}
where $\delta$ is the Dirac delta function and $$\lambda^{\pm}=\bigl(\sqrt{q(1-p)} \pm \sqrt{p(1-q)}\bigr)^2.$$
\end{proposition}
\noindent The proposition asserts that the limiting ESD of $\bm H_n$ exists and depends only on the row ratio $p$ and the column ratio $q$, provided that $\{ \bm W_{q}^{n} \}_{n \in {\mathbb N}^+}$ is concentrated. 
% In random matrix theory, similar results has been studied in \cite{wachter1980limiting} where $\bm W_q^n$ is generated from random orthogonal model, or more recently in \cite{farrell2013local} where $\bm W_q^n$ is obtained by the Kronecker product of random orthogonal matrices.
We note that the distribution $\mu_{pq}$ is exactly the same as the limiting distribution of the MANOVA ensemble.
Indeed, one can show that the MANOVA ensemble is a concentrated matrix sequence:
\begin{lemma} \label{lem:haar_uni}
Let $\bm W^{n}$ be a Haar-distributed orthogonal matrix in $\mathbb{O}(n)$ and $\bm W_{q}^{n}$ be the semi-orthogonal matrices obtained from any $qn$ (for $q \in (0,1)$) columns of $\bm W^{n}$. Then the sequence $\{\bm W_{q}^{n}\}_{n \in {\mathbb N}^+}$ is concentrated.
\end{lemma} 
\noindent Furthermore, the Kronecker product of two Haar-distributed orthogonal matrices also possesses the concentration property:
\begin{lemma} \label{lem:kron_uni}
Let $\bm U^{n_1}$ and $\bm V^{n_2}$ be Haar-distributed orthogonal matrices in $\mathbb{O}(n_1)$ and $\mathbb{O}(n_2)$, respectively. Define $\bm U_{q_1}^{n_1}$ and $\bm V_{q_2}^{n_2}$ as the semi-orthogonal matrices obtained from any $q_1$ and $q_2$ (for $q_1,q_2 \in (0,1)$) columns of $\bm U^{n_1}$ and $\bm V^{n_2}$, respectively. Then the sequence $\{\bm W_q^n = \bm U_{q_1}^{n_1} \otimes \bm V_{q_2}^{n_2}\}_{n \in {\mathbb N}^+}$ (with $q=q_1q_2$) is concentrated.
\end{lemma} 
\noindent Lemmas~\ref{lem:haar_uni} and \ref{lem:kron_uni} are immediate consequences of  Lemma~3.1 in \cite{farrell2013local}, so we omit the proof of these lemmas here.

\subsection{Proposed Estimation of the Linear Rate $\rho$}

In order to apply Proposition~\ref{prop:SSP14} to our matrix completion setting, we recall that $\bm W_{pq}^{n}$ can be viewed as the $n$-th element of a sequence of row sub-sampled matrices of $\{\bm W_{q}^{n}\}_{n \in {\mathbb N}^+}$, where $\bm W_{q}^{n} = \bm V_\perp^{n_2} \otimes \bm U_\perp^{n_1}$.
If the sequence $\{\bm W_{q}^{n}\}_{n \in {\mathbb N}^+}$ is concentrated, then (\ref{equ:ESD}) holds for $p=1-\rho_s$ and $q=(1-\rho_r)^2$. Therefore, one might expect that the smallest eigenvalue of $\bm H_n = \bm W_{pq}^{n} (\bm W_{pq}^{n})^{\topnew}$ converges to 
$$\lambda^- = \bigl(\sqrt{q(1-p)} - \sqrt{p(1-q)}\bigr)^2 .$$
Thus, by Theorem~\ref{algo:IHTSVD}, the convergence rate $\rho$ converges to $1-\lambda^-$.
The following theorem is an immediate application of Proposition~\ref{prop:SSP14} to our large-scale matrix completion setting:
\begin{theorem} \label{theo:rho_asymp}
As $m \to \infty$, assume that $\bm M$ is generated in a way that the Kronecker product $\bm W_q^n = \bm V_{\perp}^{n_2} \otimes \bm U_{\perp}^{n_1}$ forms a sequence of semi-orthogonal matrices that is concentrated. Then, provided $\rho_s \geq 1 - (1-\rho_r)^2$, the ESD $\mu_{\bm H_n}$ converges almost surely to the deterministic distribution $\mu_{\rho_r \rho_s}$ such that
\begin{align*}
d\mu_{\rho_r \rho_s} &= \Bigl(\frac{(1-\rho_r)^2-\rho_s}{1-\rho_s}\Bigr)_+ \delta(x-1) dx \\
&\quad + \frac{\sqrt{(\lambda^+-x)(x-\lambda^-)}}{2\pi (1-\rho_s)x(1-x)} \mathbb{I}[\lambda^- \leq x \leq \lambda^+] dx, \numberthis \label{equ:ESD_H}
\end{align*}
where $\lambda^{\pm}=\Bigl( \sqrt{(1-\rho_r)^2 \rho_s} \pm \sqrt{\rho_r(2-\rho_r)(1-\rho_s)} \Bigr)^2$.
\end{theorem}
\noindent Theorem~\ref{theo:rho_asymp} states the convergence of the spectral distribution of $\bm H$ as the dimensions grow to infinity. It is notable that the support of the distribution consists of the interval $[\lambda^-,\lambda^+]$ and a mass at $1$. Based on this result, we conjecture that the the smallest eigenvalue of $\bm H$ converge to $\lambda^-$ and hence, the convergence rate $\rho$ converges to $\rho_\infty$:
\begin{conjecture} \label{conj:rho_asymp}
Assume the same setting as in Theorem~\ref{theo:rho_asymp}. As $m \to \infty$, the linear rate $\rho$ defined in (\ref{equ:rho}) converges almost surely to $p_\infty = 1 - \lambda^-$, given in (\ref{equ:p_infty}).
\end{conjecture}

\section{Numerical Results}
\label{sec:exp}

In this section, we provide numerical results to verify the exact linear convergence rate of IHTSVD in (\ref{equ:rho}) with the empirical rate observed in monitoring the error through iterations. Additionally, as a supporting evidence for Theorem~\ref{theo:rho_asymp} and Conjecture~\ref{conj:rho_asymp}, we demonstrate the increasing similarity between the empirical rate and the asymptotic rate in (\ref{equ:p_infty}) as the dimensions of the matrix grow.

\subsection{Analytical Rate versus Empirical Rate}

\begin{figure}
    \centering
    \includegraphics[scale=.6]{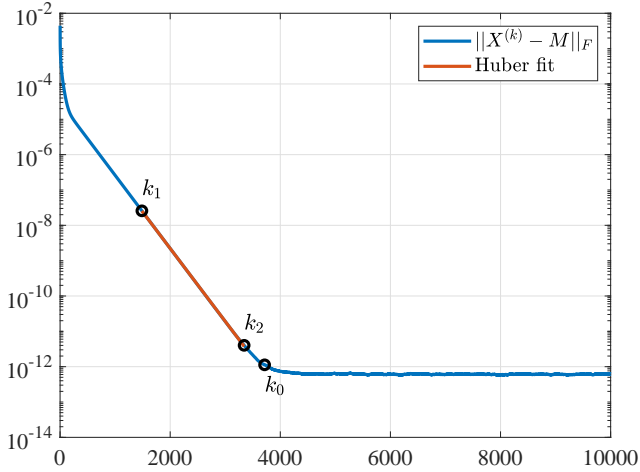}
    \caption{Estimation of the empirical rate using the error sequence $\{\norm{\bm X^{(k)}-\bm M}_F\}_{k=k_1}^{k_2}$. Due to the numerical error below $10^{-12}$, we need to identify the `turning point' at $k_0$ and then set $k_1=\lfloor 0.4 k_0 \rfloor$ and $k_2=\lfloor 0.9 k_0 \rfloor$.}
    \label{fig:estimate}
\end{figure}

\begin{figure*}
    \centering
    \begin{subfigure}[b]{0.45\textwidth}
        \centering
        \includegraphics[width=\textwidth]{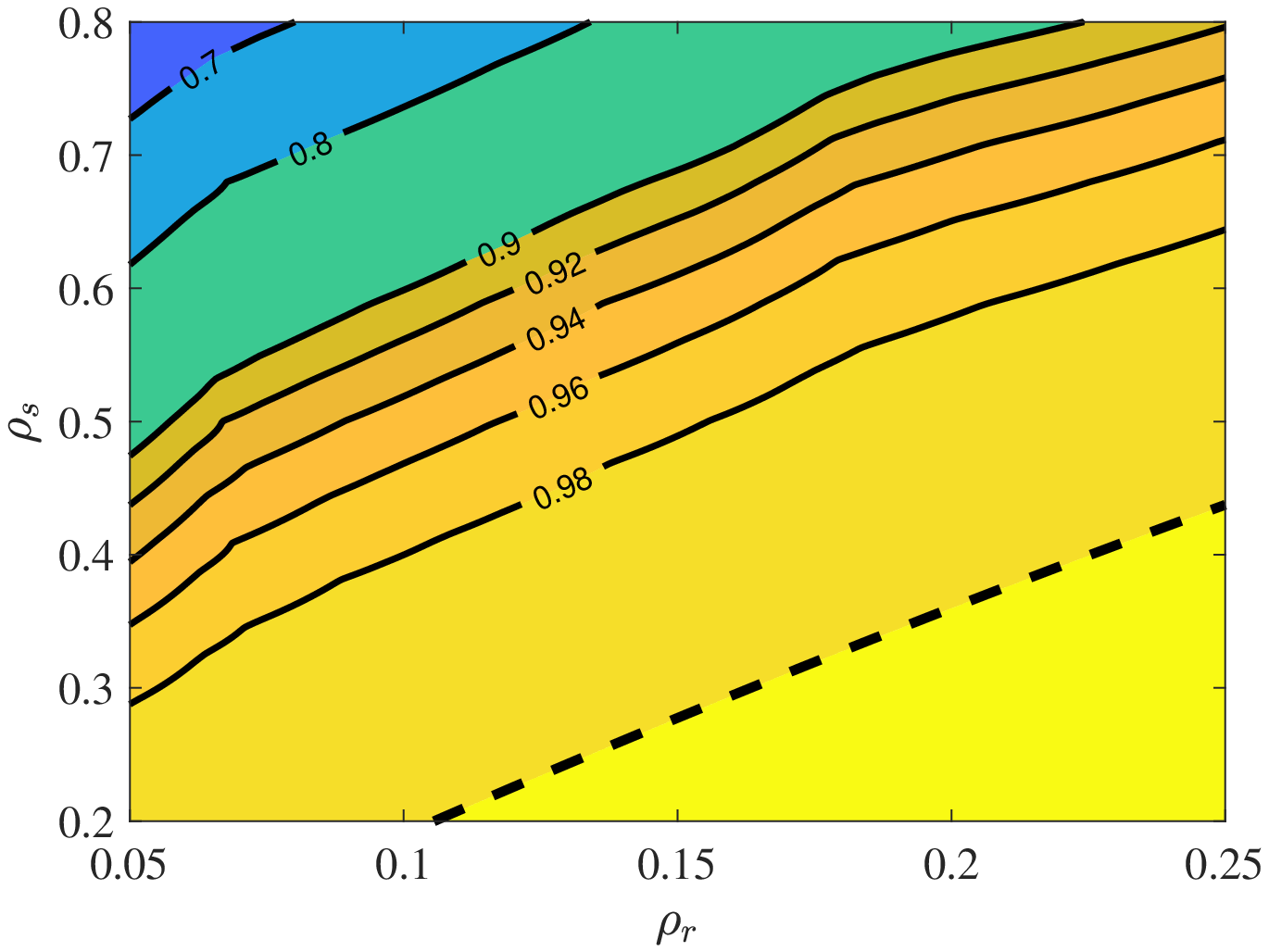}
        \caption{Analytical rate}
    \end{subfigure}
    \hfill
    \begin{subfigure}[b]{0.45\textwidth}
        \centering
        \includegraphics[width=\textwidth]{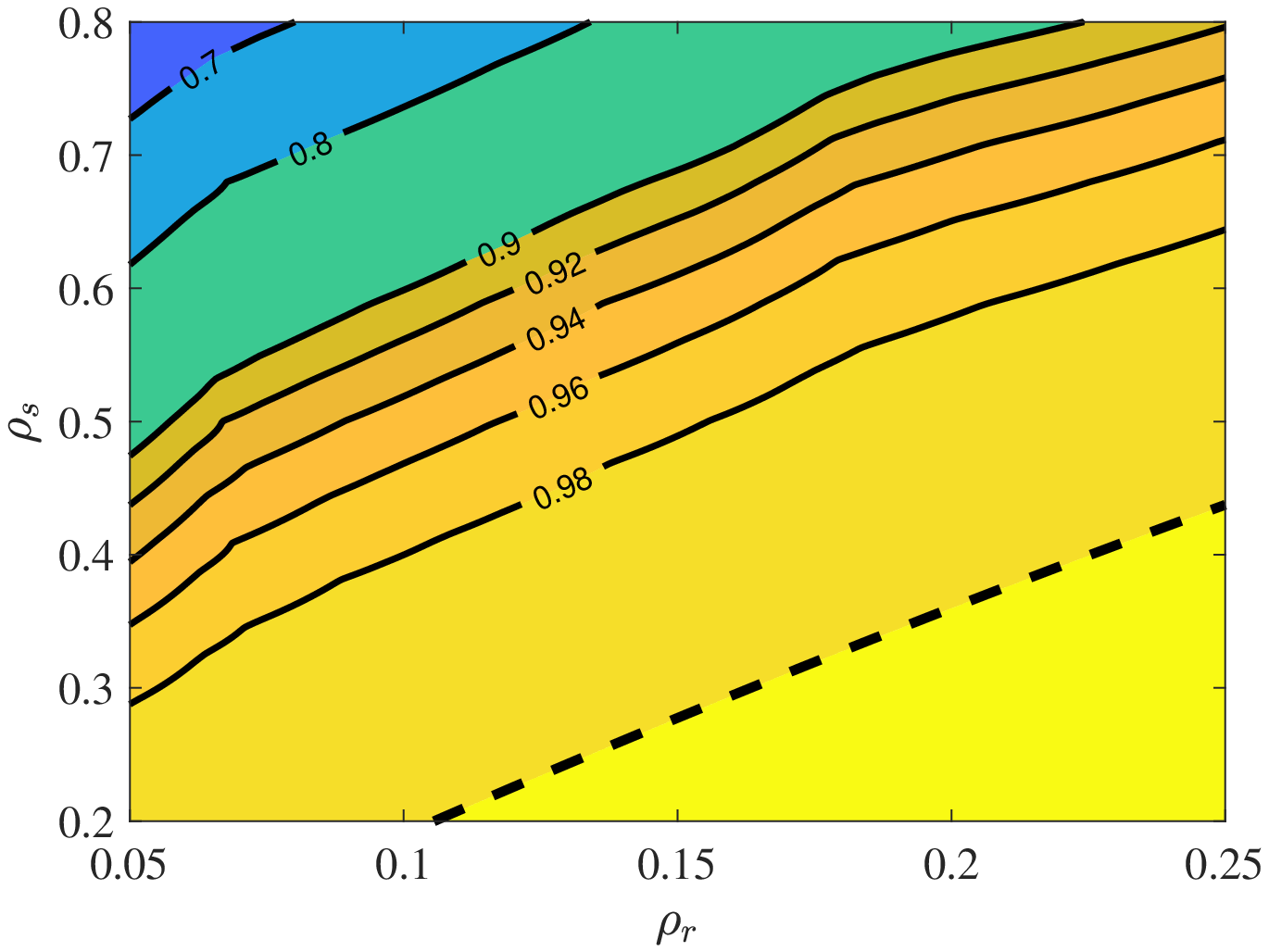}
        \caption{Empirical rate}
    \end{subfigure}
    \hfill
    \begin{subfigure}[b]{0.45\textwidth}
        \centering
        \includegraphics[width=\textwidth]{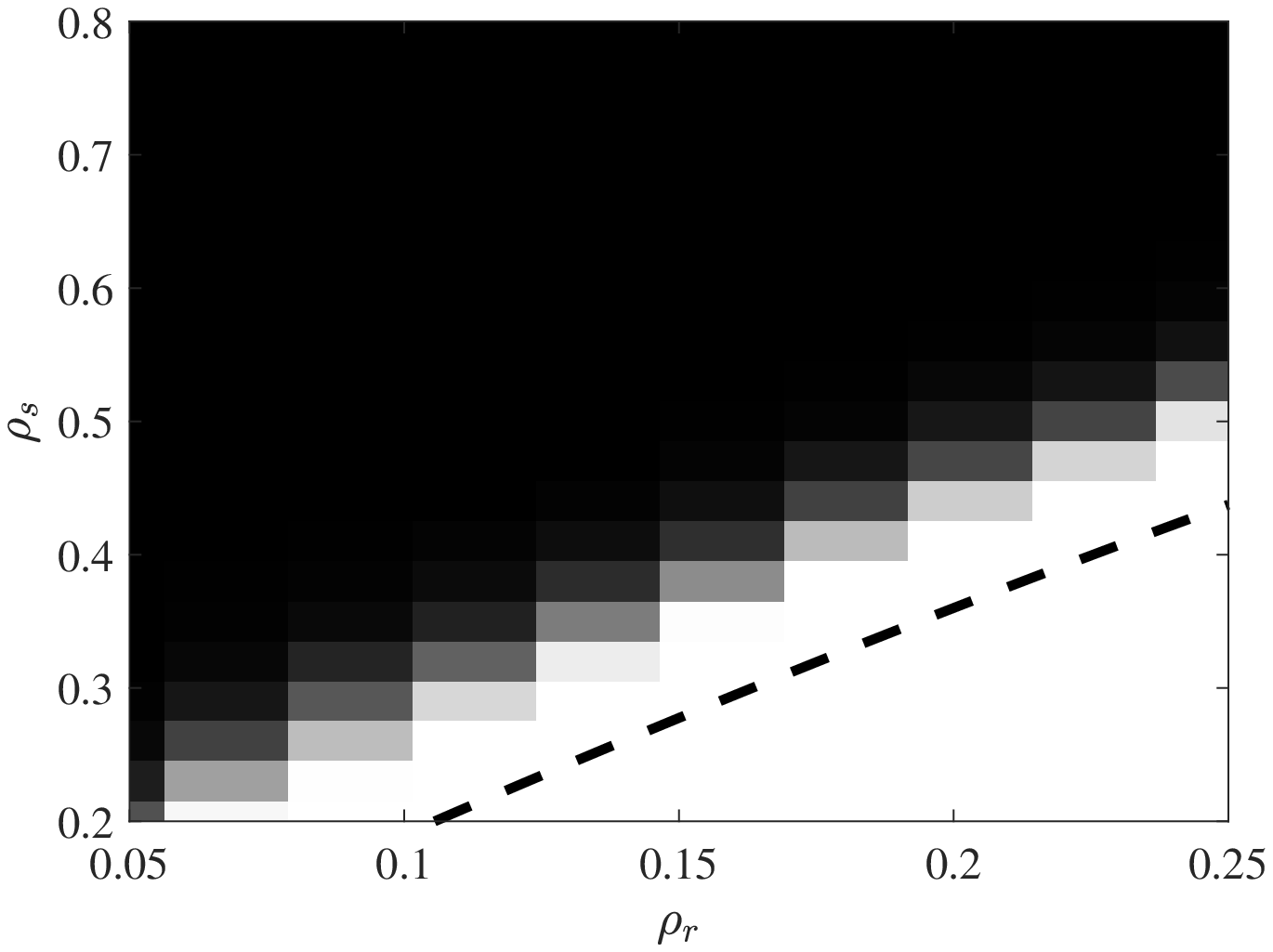}
        \caption{Probability of linear convergence (analytical)}
    \end{subfigure}
    \hfill
    \begin{subfigure}[b]{0.45\textwidth}
        \centering
        \includegraphics[width=\textwidth]{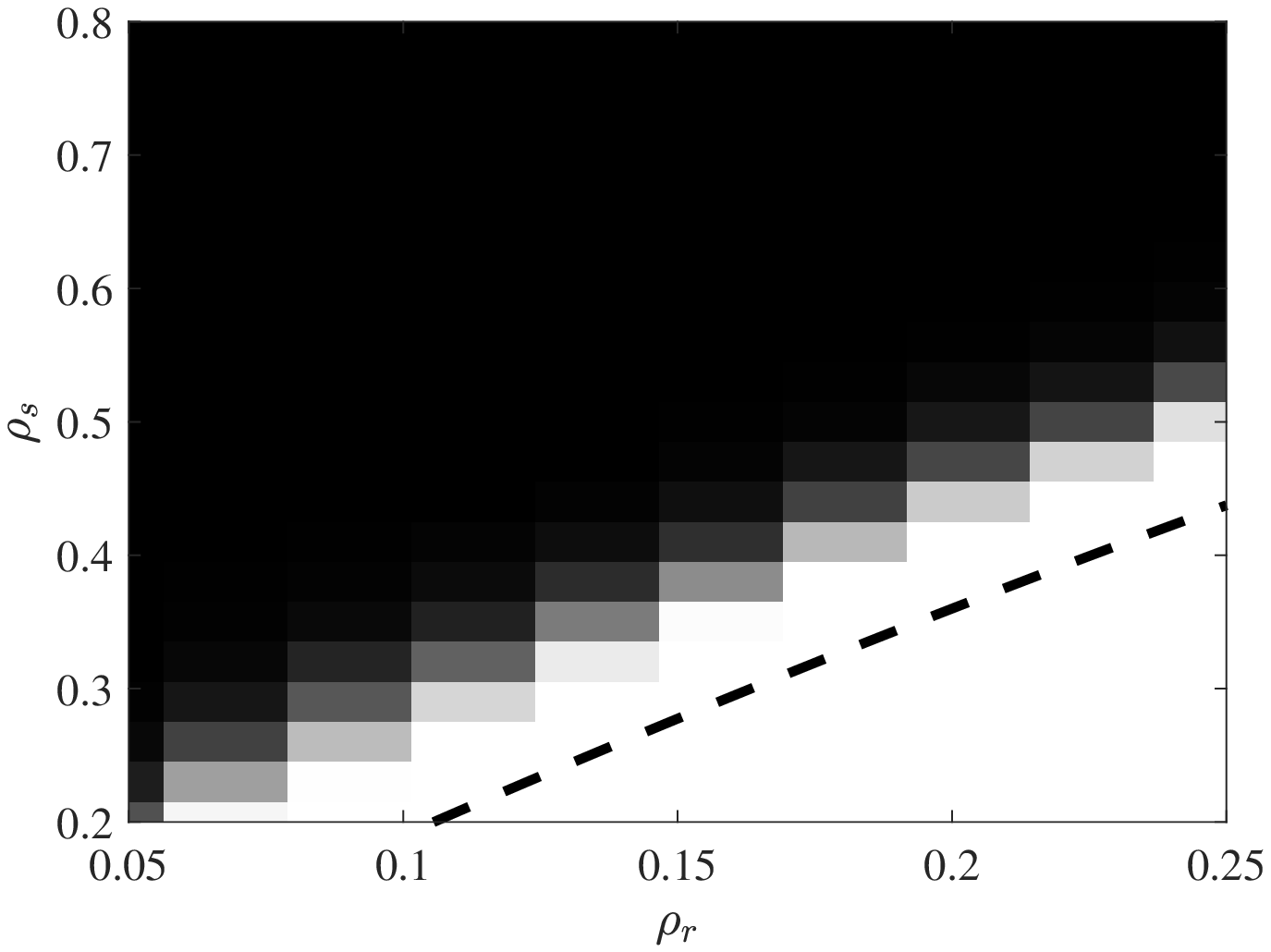}
        \caption{Probability of linear convergence (empirical)}
    \end{subfigure}
    \hfill
    \caption{The analytical rate and the empirical rate of convergence of IHTSVD as a function of the relative rank $\rho_r$ and the sampling ratio $\rho_s$, with $n_1=50$ and $n_2=40$. (a) Contour plot of the analytical rate as a function of $\rho_r$ and $\rho_s$. (b) Contour plot of the empirical rate as a function of $\rho_r$ and $\rho_s$. (c) Empirical probability of linear convergence based on the analytical rate. (d) Empirical probability of linear convergence based on the empirical rate. In (c) and (d), the black color corresponds to linear convergence, whereas the white color corresponds to no linear convergence. The data is evaluated based on a $12 \times 21$ grid over $\rho_r$ and $\rho_s$ and the value of each point in the grid is averaged over $1000$ runs. Additionally, a dashed line is included in each plot to indicate the line $1-\rho_s = (1-\rho_r)^2$. The similarity between the left column and the right column demonstrates the utility of the empirical rate in estimating/approximating the analytical rate.}
    \label{fig:analytic}
\end{figure*}

In this experiment, we verify the analytical expression of the linear convergence rate of IHTSVD by comparing it with the empirical rate obtained by measuring the decrease in the norm of the error matrix. Our goal is to demonstrate that they agree in various settings of $\rho_r$ and $\rho_s$.

\vspace{5pt}
\noindent \textbf{Data generation.} 
We first set the dimensions $n_1=50$ and $n_2=40$. Next, for each $r$ in $\{1,2,\ldots,12\}$, we generate the rank-$r$ matrix $\bm M$ as follows. We construct the random orthogonal matrices $\bm U$ and $\bm V$ by (i) generating a $n_1 \times n_2$ random matrix whose entries are $i.i.d$ normally distributed ${\cal N}(0,1)$ and (ii) performing the singular value decomposition of the resulting matrix. The matrices $\bm U$ and $\bm V$ are comprised of the corresponding left and right singular vectors.
Then, the rank-$r$ matrix $\bm M$ is generated by taking the product $\bm U_1 \bm \Sigma_1 \bm V_1^{\topnew}$, where $\bm \Sigma_1 = \diag(r,r-1,\ldots,1)$ and $\bm U_1, \bm V_1$ are the first $r$ columns of $\bm U$ and $\bm V$, respectively. 
Finally, for each $s$ in the linearly spaced set $\{ 0.2n, 0.23n,0.26n,\ldots, 0.8n \}$, we create the $1000$ different sampling sets, each of them is obtained by generating a random permutation of the set $[n]$ and then selecting the first $s$ elements of the permutation. Thus, we obtain a $12 \times 21$ grid based on the values of $r$ and $s$ such that (i) grid points corresponding to the same rank $r$ share the same underlying matrix $\bm M$; (ii) each point on the grid corresponds to $1000$ different sampling sets.
% \begin{align*}
%     r = \Bigl\lceil \frac{1}{2} \bigl( n_1+n_2 - \sqrt{(n_1+n_2)^2 - 4n_1 n_2(1-(1-\rho_r)^2)} \bigr) \Bigr\rceil .
% \end{align*}

\vspace{5pt}
\noindent \textbf{Estimating Analytical Rate and Empirical Rate.} 
We calculate the analytical rate for each aforementioned setting of $\bm M$ and $\Omega$ using (\ref{equ:rho}). Due to numerical errors in computing small eigenvalues, we need to set all the resulting rates that are greater than $1$ to $1$, indicating there is no linear convergence in such cases. For the calculation of the empirical rate, we run Algorithm~\ref{algo:IHTSVD} in the same setting with $K=10000$ iterations. The initial point $\bm X^{(0)}$ is obtained by adding $i.i.d.$ normally distributed noise with standard deviation $\sigma=10^{-4}$ to the entries of $\bm M$. 
% \hl{RR: note that some picky readers/reviewers may ask about this choice. Are we sure that the resulting initialization fits within the ROC? Alternatively, maybe this leads to a small l2 error - smaller than the radius of the ROC. In this case, the criticism could be that we don't verify the ROC. TV: don't have ROC for large matrices}
Here we note that $\sigma$ is chosen to be small for two reasons: (i) for large matrices, even small $\sigma$ for individual entry can add up to a large error on the entire matrix; and (ii) while the cost of computing $\lambda_{\min}$ (and hence, the region of convergence) is prohibitively expensive for large matrices, choosing small $\sigma$ empirically guarantees the initialization is inside the region of convergence.

Next, we record the error sequence $\{\norm{\bm X^{(k)} - \bm M}_F\}_{k=1}^K$ and determine if the algorithm converges linearly to $\bm M$ by checking whether there exists $\hat{K} \leq K$ such that $\norm{\bm X^{(\hat{K})} - \bm M}_F < \epsilon \norm{\bm X^{(0)} - \bm M}_F$, for $\epsilon=10^{-8}$. If the relative error is above $\epsilon$, we set the empirical rate to $1$ to indicate that the algorithm does not converge linearly. 
However, it is important to note that this heuristic \textbf{does not perfectly} detect linear convergence since it overlooks the case in which the linear rate is extremely close to $1$ and it requires more than $K=10000$ iterations to reach a relative error below $\epsilon$. As can be seen later, to compromise this computational limit, we resort to setting the analytical rate that is greater than $0.998$ to $1$ when making a comparison between the analytical rate and the empirical rate\footnote{Substituting $\epsilon=10^{-8}$ and $K(\epsilon)=10000$ into (\ref{equ:Ne}) and assuming the constant $c$ is negligible, we obtain $\lambda_{\min}(\bm H) \approx 1.8 \times 10^{-3}$, which in turn implies $\rho = 1 - \lambda_{\min}(\bm H) = 0.998$.}.
In case the relative error is less than $\epsilon$, we terminate the algorithm at the $\hat{K}$-th iteration (early stop) and perform a simple fitting for an exponential decrease on the error sequence $\{\norm{\bm X^{(k)} - \bm M}_F\}_{k=1}^{\hat{K}}$ to obtain the empirical rate.

After obtaining the analytical rate and the empirical rate over the 2-D grid, we report the result in the contour plots of the rate as a function of $\rho_r$ and $\rho_s$ in Fig.~\ref{fig:analytic}-(a) and Fig.~\ref{fig:analytic}-(b). Since our original grid is non-uniform, we perform a scattered data interpolation, which uses a Delaunay triangulation of the scattered sample points to perform interpolation \cite{amidror2002scattered}, to evaluate the rate over a $1001 \times 1001$ uniform grid based on $\rho_r$ and $\rho_s$. Due to the aforementioned limitation of estimating the empirical rate, we apply a threshold of $0.998$ to both of the interpolated data for the analytical rate and the empirical rate, setting any value above the threshold to $1$. 

Finally, at each point of the $12 \times 21$ grid, we calculate the probability of linear convergence over $1000$ runs.
For the analytical rate, the linear convergence is determined by checking whether $\lambda_{\min}(\bm H)<1$.
For the empirical rate, we use the aforementioned discussion on determining weather the algorithm converges linearly with $K=10000$ and $\epsilon=10^{-8}$.
The results are visualized in Fig.~\ref{fig:analytic}-(c) and Fig.~\ref{fig:analytic}-(d).

\vspace{5pt}
\noindent \textbf{Results.} 
Given the values of the analytical rate and the empirical rate of $1000$ matrix completion settings for each point on the $12 \times 21$ grid, the mean squared difference between the two rates in our experiment is $2.9659 \times 10^{-5}$.
Figure~\ref{fig:analytic} illustrates the similarity between the analytical rate and the empirical rate evaluated under various settings of matrix completion. In both Fig.~\ref{fig:analytic}-(a) and Fig.~\ref{fig:analytic}-(b), we observed a matching behavior as in Fig.~\ref{fig:asymptotic101}: smaller rank and more observation result in faster linear convergence of IHTSVD. 
However, the contour lines in Fig.~\ref{fig:analytic} are not as smooth as those with asymptotic behavior in Fig.~\ref{fig:asymptotic101} due to the resolution of the grid as well as the large variance of the convergence rate under different sampling patterns when $n_1$ and $n_2$ are relatively small.
On the other hand, it can be seen from Fig.~\ref{fig:analytic}-(c) and Fig.~\ref{fig:analytic}-(d) that there are a linear-convergence area (black) above the boundary line at $1-\rho_s = (1-\rho_r)^2$ and a no-linear-convergence area (white) below the boundary line. The transition area (gray) near above the boundary line corresponds to the settings in which some sampling sets yield $\lambda_{\min}(\bm H)=0$ while some other sampling sets yield $\lambda_{\min}(\bm H) > 0$. We discuss this transition region further in the next experiment.

To conclude, note that in order to obtain the analytical rate, we need to compute the smallest eigenvalue of a $(n - s) \times (n - s)$ matrix, which is computationally expensive for large $n=n_1n_2$. In particular, when $s=\O(n)$, the cost of computing the analytical rate is $\O(n^3)$. On the other hand, the empirical rate offers an alternative but more efficient way to estimate the convergence rate via running Algorithm~\ref{algo:IHTSVD} whose computational complexity per iteration is $\O(nr)$. As a by-product, our proposed empirical rate can be used to efficiently estimate the smallest eigenvalue of the large matrix $\bm H$.

\subsection{Non-asymptotic Rate versus Asymptotic Rate}

\begin{figure*}
    \flushleft
    \begin{subfigure}[b]{0.31\textwidth}
        \centering
        \includegraphics[width=\textwidth]{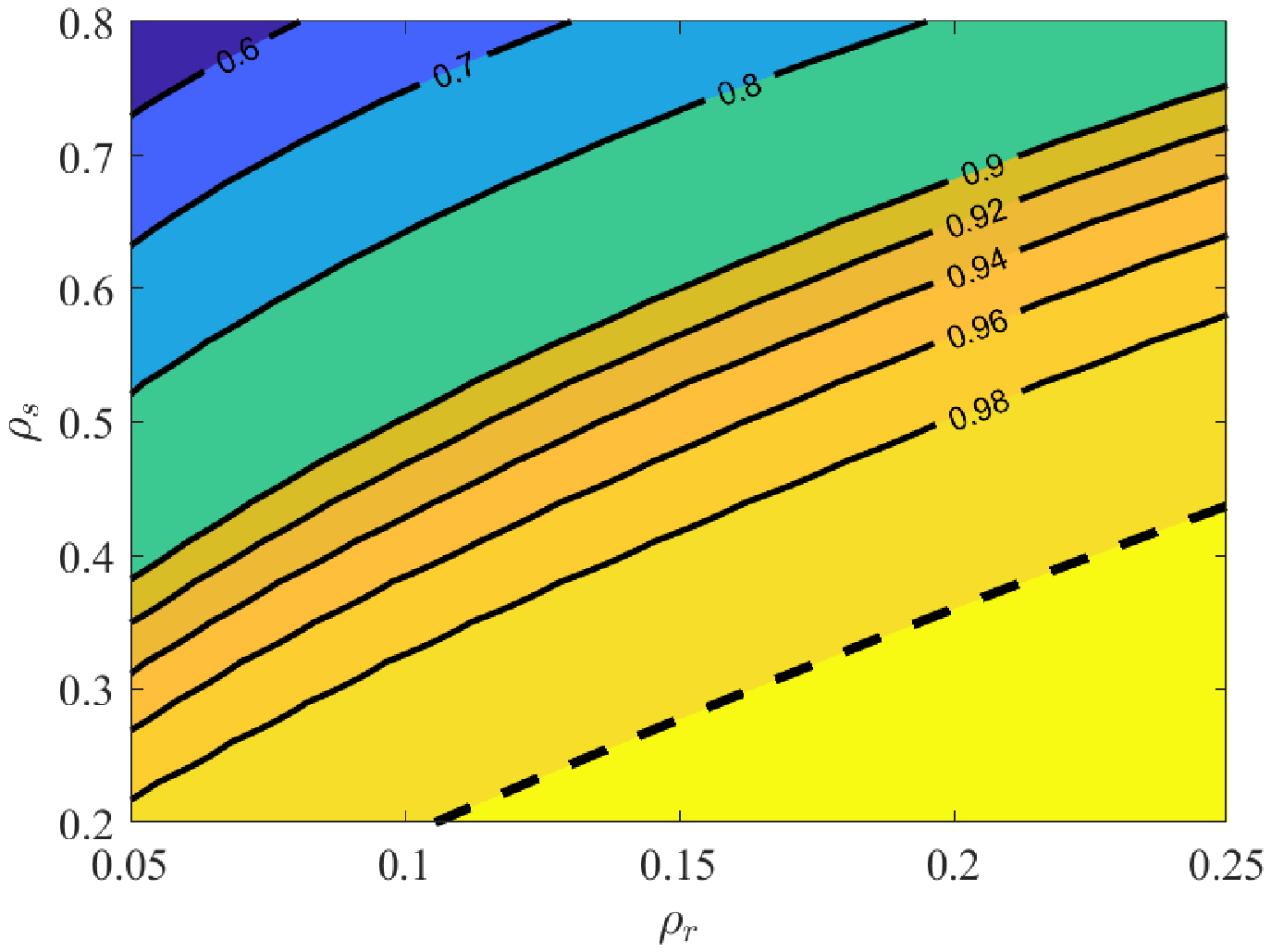}
        \caption{$n_1=500, n_2=400$}
    \end{subfigure}
    \quad
    \begin{subfigure}[b]{0.31\textwidth}
        \centering
        \includegraphics[width=\textwidth]{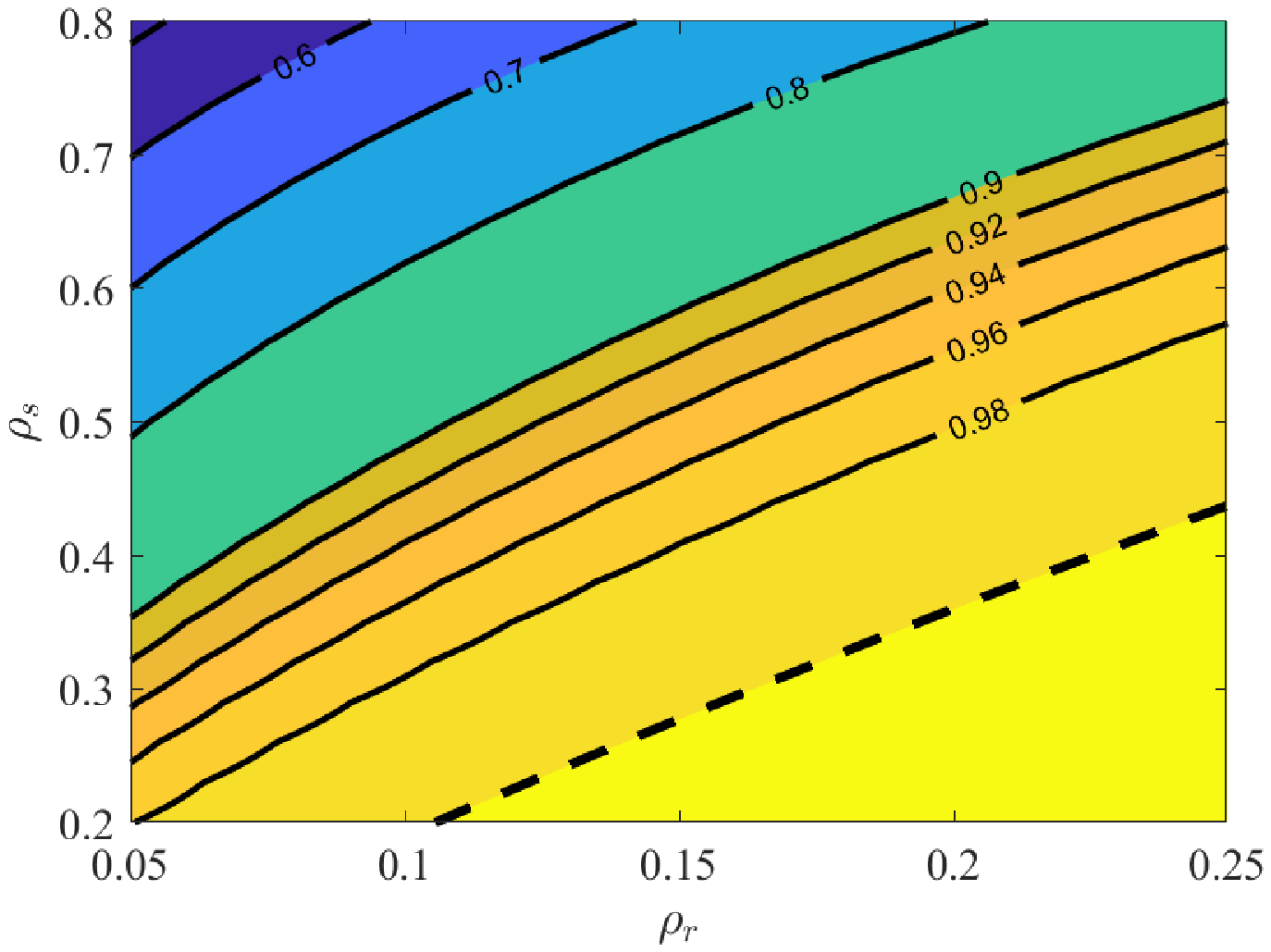}
        \caption{$n_1=1200, n_2=1000$}
    \end{subfigure}
    \quad
    \begin{subfigure}[b]{0.31\textwidth}
        \centering
        \includegraphics[width=\textwidth]{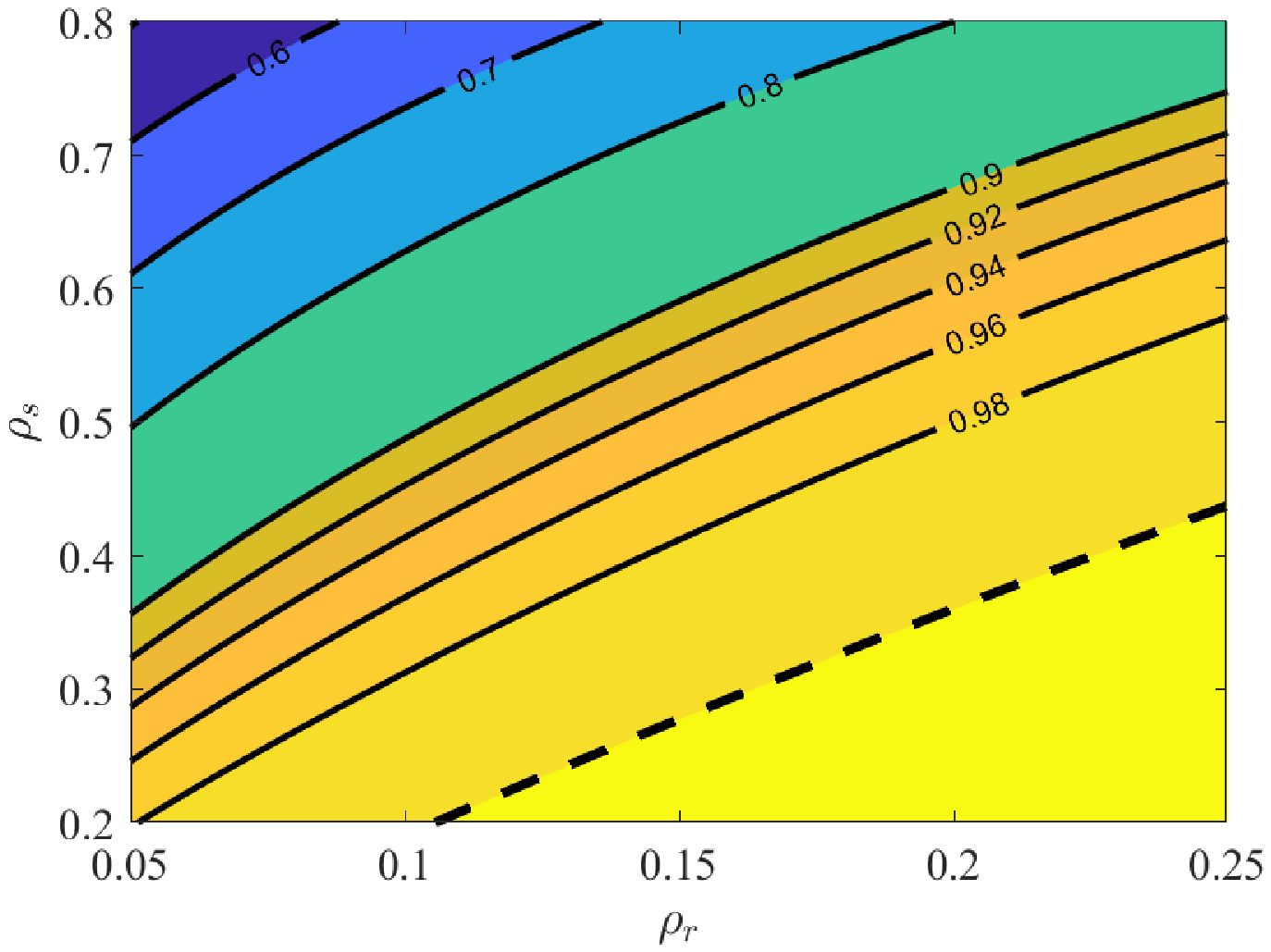}
        \caption{$p_\infty$ (zoom in)}
    \end{subfigure}
    
    % \medskip
    % \flushleft
    
    \begin{subfigure}[b]{0.31\textwidth}
        \centering
        \includegraphics[width=\textwidth]{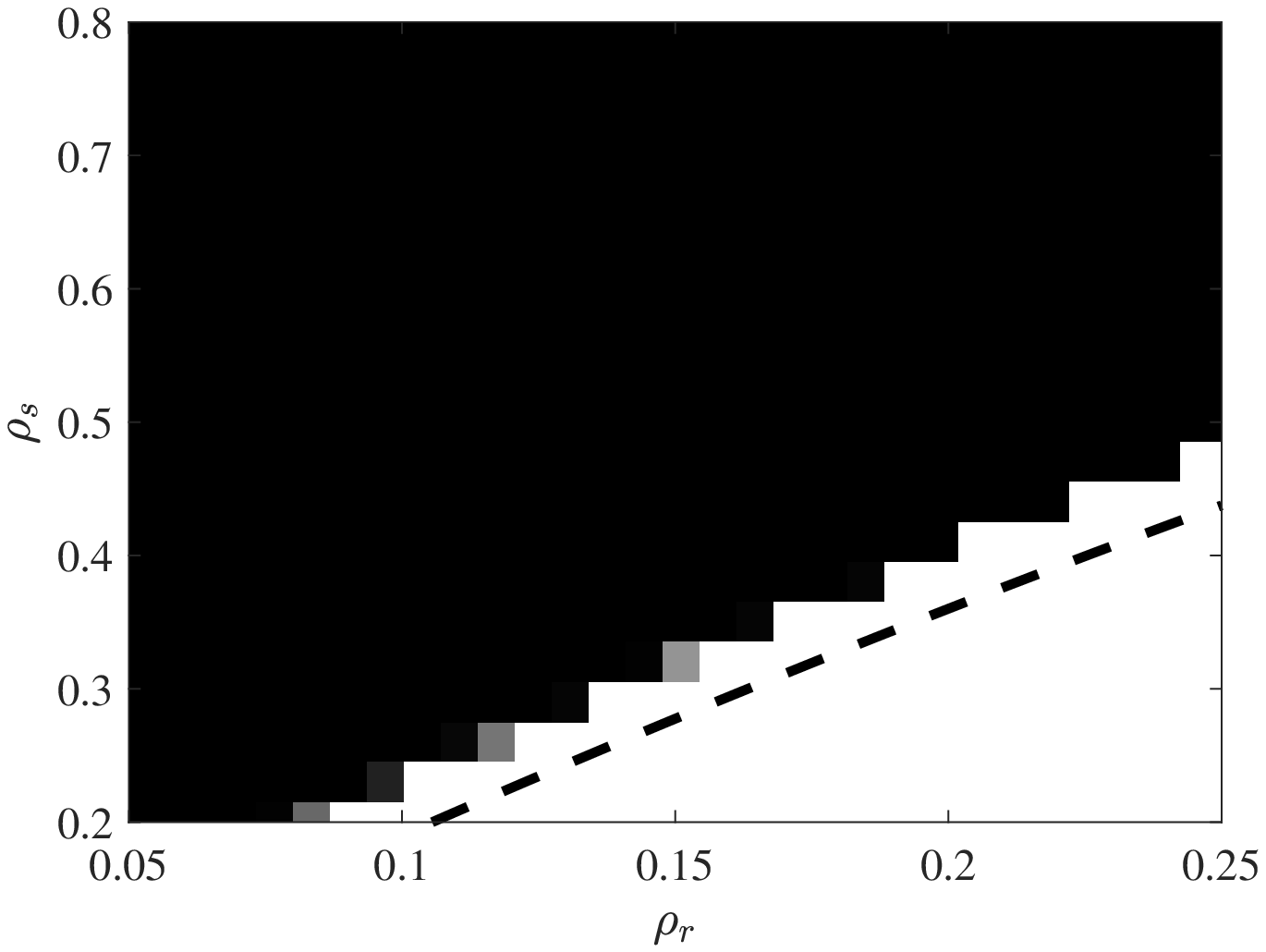}
        \caption{$n_1=500, n_2=400$}
    \end{subfigure}
    \quad
    \begin{subfigure}[b]{0.31\textwidth}
        \centering
        \includegraphics[width=\textwidth]{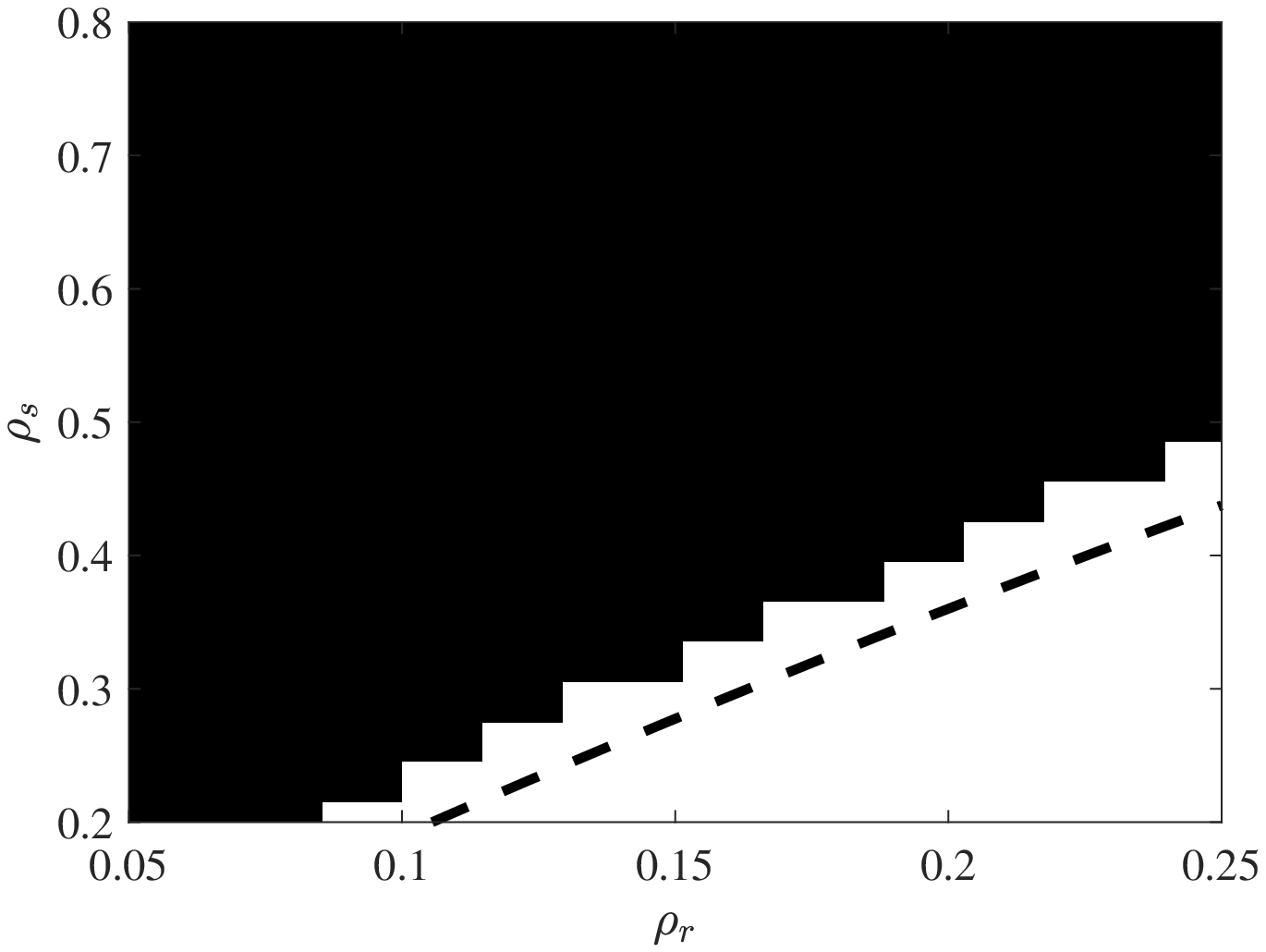}
        \caption{$n_1=1200, n_2=1000$}
    \end{subfigure}
    \quad
    \begin{subfigure}[b]{0.31\textwidth}
        \centering
        \includegraphics[width=\textwidth]{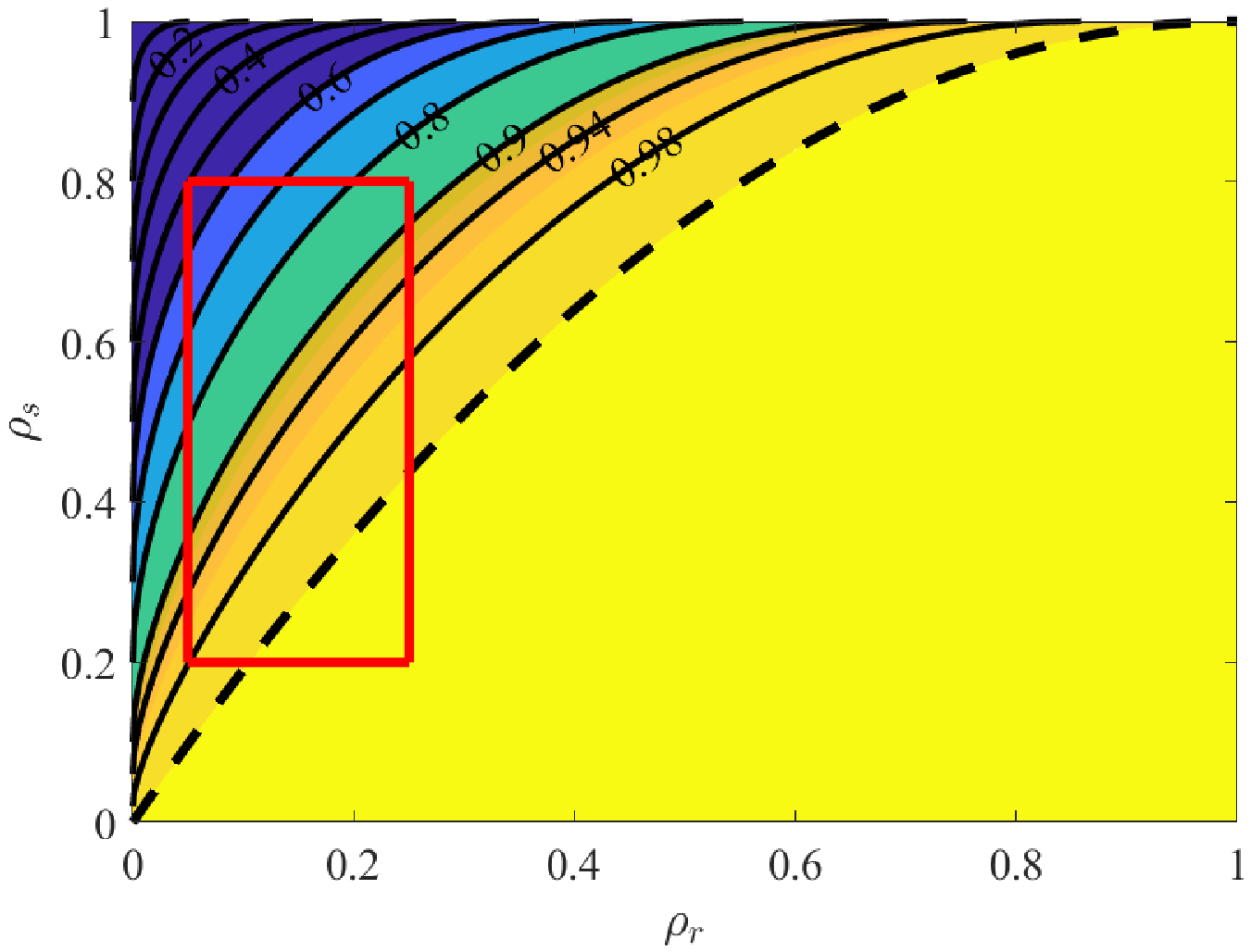}
        \caption{$p_\infty$ (zoom out)}
    \end{subfigure}
    \hfil
    
    \caption{The empirical rate and the asymptotic rate of convergence of IHTSVD as a function of the relative rank $\rho_r$ and the sampling ratio $\rho_s$. (a) Contour plot of the empirical rate as a function of $\rho_r$ and $\rho_s$ for $n_1=500, n_2=400$. (b) Contour plot of the empirical rate as a function of $\rho_r$ and $\rho_s$ for $n_1=1200, n_2=1000$. (c) A zoom-in contour plot of the asymptotic rate as a function of $\rho_r$ and $\rho_s$. (d) Empirical probability of linear convergence based on the empirical rate in (a). (e) Empirical probability of linear convergence based on the empirical rate in (b).\protect\footnotemark ~(f) A zoomed-out contour plot of the asymptotic rate as a function of $\rho_r$ and $\rho_s$. The red solid rectangular corresponds to the zoomed-in region in (c). The data is evaluated based on 2-D grids over $\rho_r$ and $\rho_s$ and the value of each point in each grid is averaged over $100$ runs. Additionally, a dashed line is included in each plot to indicate the line $1-\rho_s = (1-\rho_r)^2$. The striking similarity between plots (b) and (c) illustrates the utility of our convergence rate analysis in large-scale settings.}
    \label{fig:asymp}
\end{figure*}

In this experiment, we compare the asymptotic rate given in Theorem~\ref{theo:rho_asymp} with the convergence rate of IHTSVD for large-scale matrix completion. For convenience, we refer the latter as the non-asymptotic rate. As mentioned, we use the empirical rate instead of the analytical rate to estimate the non-asymptotic rate due to the computational efficiency. 

\vspace{5pt}
\noindent \textbf{Data generation.}
We consider two settings of $(n_1,n_2)$, i.e., $n_1=500, n_2=400$ and $n_1=1200, n_2=1000$.
Similar to the previous experiment, we generate $\bm M$ and $\Omega$ based on a 2-D grid over $r$ and $s$. While the values of $s$ are still selected from the set $\{ 0.2n, 0.23n,0.26n,\ldots, 0.8n \}$, the values of $r$ are chosen differently for each setting of $(n_1,n_2)$. In particular, for $n_1=500, n_2=400$, we select the values of $r$ from the linearly spaced set $\{1,4,7,\ldots,118\}$. For $n_1=1200, n_2=1000$, we select the values of $r$ from the linearly spaced set $\{1,9,17,\ldots,297\}$. 
Thus, in the former setting, the grid size is $40 \times 21$, while in the latter setting, the grid size is $38 \times 21$. We note that both grids are non-uniform in terms of $\rho_r$ and $\rho_s$.

\vspace{5pt}
\noindent \textbf{Implementation.}
The calculations of the empirical rate and the probability of linear convergence are the same as the previous experiment. For computational efficiency, we omit the points on the grid that are below the boundary line $1-\rho_s = (1-\rho_r)^2$, i.e., $s<(n_1+n_2-r)r$, since it is evident that there is no linear convergence guaranteed at these points. No analytical rate is given in this experiment because calculating the smallest eigenvalue of a $(n_1n_2-s) \times (n_1n_2-s)$ matrix is computationally expensive for large $n_1$ and $n_2$. On the other hand, the contour plot of the asymptotic rate is straightforward to obtain using (\ref{equ:p_infty}).

\footnotetext{In (d) and (e), the black color corresponds to linear convergence, whereas the white color corresponds to no linear convergence.}

\vspace{5pt}
\noindent \textbf{Results.}
In Fig.~\ref{fig:asymp}-(a) and Fig.~\ref{fig:asymp}-(b), we present the average empirical rate of linear convergence of IHTSVD as a function of the relative rank and the sampling rate in two large-scale settings. Observing  the average empirical rate from Fig.~\ref{fig:analytic}-(b) to Fig.~\ref{fig:asymp}-(a) and to Fig.~\ref{fig:asymp}-(b) as the dimensions increase, we note a shift of the contour lines towards the bottom-right corner, approaching those of the asymptotic rate in Fig.~\ref{fig:asymp}-(c).
This matches our intuition from Theorem~\ref{theo:rho_asymp} that as the dimensions grow to infinity, the linear rate of IHTSVD converges to the asymptotic rate $p_\infty$.
Additionally, from Fig.~\ref{fig:analytic}-(d), Fig.~\ref{fig:asymp}-(d), and Fig.~\ref{fig:asymp}-(e), we observe that the linear-convergence area (black) becomes larger in larger matrix completion settings, indicating the isoline at $0.998$ approaches closer to the line $1-\rho_s = (1-\rho_r)^2$ (dashed line).
It is notable, however, that the transition between the linear-convergence area and the no-linear-convergence area is more abrupt as the dimensions increases.
This phenomenon also matches our intuition in Conjecture~\ref{conj:rho_asymp}, indicating that there is smaller variance in the empirical rate in large-scale settings, with respect to different random sampling patterns on the same underlying matrix.\footnote{Another evidence supporting this argument is the comparison of the coefficient of variation of the empirical rate in Fig.~\ref{fig:asymp}-(a) and Fig.~\ref{fig:asymp}-(b). We provide the detail in Fig.~\ref{fig:var} in the Supplementary Material.}

% \hl{reemphasize that this is our conjecture and theory need to be developed, and the variance of the rate plot and compare}

\section{Conclusions and Future Work}
\label{sec:conc}

In this paper, we established a closed-form expression of the linear convergence rate of an iterative hard thresholding method for solving matrix completion. We also identified the local region around the solution that guarantees the convergence of the algorithm. Furthermore, in large-scale settings, we leveraged the result from random matrix theory to offer a simple estimation of the asymptotic convergence rate in practice. Under certain assumptions, we showed that the convergence rate of IHTSVD converges almost surely to our proposed estimate.

In future work, we would like to extend our local convergence analysis to other IHT methods with different step sizes, e.g., SVP \cite{jain2010guaranteed} and accelerated IHT \cite{vu2019local,vu2019accelerating}.
Moreover, it would be interesting to study the non-asymptotic behavior of the convergence rate in large-scale settings. Finally, we believe the technique presented in this manuscript can be applied to study the local convergence of other non-convex methods such as alternating minimization \cite{jain2013low} and gradient descent \cite{sun2016guaranteed}.

% \vfill

% or
% \appendix  % for no appendix heading


% \newpage
% \appendices

\appendix[Proof of Theorem~\ref{theo:IHT}]
\label{appdx:IHT}

\subsection{Proof of Lemma~\ref{lem:error_recur}}
By the definition of the error matrix, we have
\begin{align*}
&\bm E^{(k+1)} = \bm X^{(k+1)} - \bm M \\
&= \Bigl( \P_{\bar{\Omega}} \bigl( \P_r (\bm X^{(k)}) \bigr) + \P_{\Omega}(\bm M) \Bigr) - \bigl( \P_\Omega(\bm M) + \P_{\bar{\Omega}}(\bm M) \bigr) \\
&= \P_{\bar{\Omega}} \bigl( \P_r (\bm M + \bm E^{(k)}) - \bm M \bigr) . \numberthis \label{equ:Ek1}
\end{align*}
From Proposition~\ref{prop:rankop}, we can reorganize (\ref{equ:rankop}) to obtain
\begin{align*}
    \P_r (\bm M + \bm E^{(k)}) - \bm M = \bm E^{(k)} - \bm P_{\bm U_{\perp}} \bm E^{(k)} \bm P_{\bm V_{\perp}} + \bm R(\bm E^{(k)}) .
\end{align*}
Substituting the last equation back into (\ref{equ:Ek1}) yields the recursion on the error matrix as in (\ref{equ:E}).
% Notice that the error matrix satisfies $\bm E^{(k)} = [\bm E^{(k)}]_{\bar{\Omega}}$.

Next, let us denote $\bm e^{(k)} = {\bm S}_{\bar{\Omega}}^{\topnew} \vect (\bm E^{(k)})$, for $k=1,2,\ldots$.
Vectorizing equation (\ref{equ:E}) and left-multiplying both sides with ${\bm S}_{\bar{\Omega}}$ yield
\begin{align*}
&\bm e^{(k+1)} = {\bm S}_{\bar{\Omega}}^{\topnew} \vect \Bigl( \P_{\bar{\Omega}} \bigl( \bm E^{(k)} - \bm P_{\bm U_{\perp}} \bm E^{(k)} \bm P_{\bm V_{\perp}} + \bm R(\bm E^{(k)}) \bigr) \Bigr) .
\end{align*}
Using the property of selection matrices in Definition~\ref{def:S}, we further have
\begin{align*}
\bm e^{(k+1)} &= {\bm S}_{\bar{\Omega}}^{\topnew} {\bm S}_{\bar{\Omega}} {\bm S}_{\bar{\Omega}}^{\topnew} \vect \bigl( \bm E^{(k)} - \bm P_{\bm U_{\perp}} \bm E^{(k)} \bm P_{\bm V_{\perp}} + \bm R(\bm E^{(k)}) \bigr) \\
&= {\bm S}_{\bar{\Omega}}^{\topnew} \vect \bigl( \bm E^{(k)} - \bm P_{\bm U_{\perp}} \bm E^{(k)} \bm P_{\bm V_{\perp}} + \bm R(\bm E^{(k)}) \bigr) .
\end{align*}
Since $\vect(\bm P_{\bm U_{\perp}} \bm E^{(k)} \bm P_{\bm V_{\perp}}) = (\bm P_{\bm V_{\perp}} \otimes \bm P_{\bm U_{\perp}}) \vect(\bm E^{(k)})$, the last equation can be represented as
\begin{align*}
\bm e^{(k+1)} = {\bm S}_{\bar{\Omega}}^{\topnew} \vect(\bm E^{(k)}) - {\bm S}_{\bar{\Omega}}^{\topnew} &(\bm P_{\bm V_{\perp}} \otimes \bm P_{\bm U_{\perp}}) \vect(\bm E^{(k)}) \\
&+ {\bm S}_{\bar{\Omega}}^{\topnew} \vect \bigl( \bm R(\bm E^{(k)}) \bigr) . \numberthis \label{equ:Eks}
\end{align*}
On the other hand, (\ref{equ:E}) implies, for any $k \geq 1$, $\bm E^{(k)} = \P_{\bar{\Omega}} (\bm E^{(k)})$ and
\begin{align*}
    \vect(\bm E^{(k)}) = \vect\bigl( \P_{\bar{\Omega}} (\bm E^{(k)}) \bigr) = {\bm S}_{\bar{\Omega}} {\bm S}_{\bar{\Omega}}^{\topnew} \vect(\bm E^{(k)}) = {\bm S}_{\bar{\Omega}} \bm e^{(k)} .
\end{align*}
Substituting the last equation into the RHS of (\ref{equ:Eks}) yields (\ref{equ:e}). 

% \hfill $\square$

\subsection{Proof of Lemma~\ref{lem:error_ineq}}

Applying the triangle inequality to the RHS of (\ref{equ:e}) yields
\begin{align} \label{equ:einq0}
    \norm{\bm e^{(k+1)}}_2 \leq &\norm{(\bm I - \bm H) \bm e^{(k)} }_2 + \norm{\bm r( \bm e^{(k)})}_2 ,
\end{align}
where we recall $\bm H = {\bm S}_{\bar{\Omega}}^{\topnew} (\bm P_{\bm V_{\perp}} \otimes \bm P_{\bm U_{\perp}}) {\bm S}_{\bar{\Omega}}$.
By the definition of the operator norm, we have 
\begin{align*}
    \norm{(\bm I - \bm H) \bm e^{(k)} }_2 &\leq \norm{\bm I - \bm H}_2 \norm{\bm e^{(k)}}_2 \\
    &= \max_{i} \bigl\{ \abs{1-\lambda_{i}(\bm H)} \bigr\} \cdot \norm{\bm e^{(k)}}_2 \\
    &= \bigl(1-\lambda_{\min}(\bm H)\bigr) \norm{\bm e^{(k)}}_2 \numberthis \label{equ:einq1} ,
\end{align*}
where the last equality stems from the fact that all eigenvalues of $\bm H$ lie between $0$ and $1$.
From (\ref{equ:einq0}) and (\ref{equ:einq1}), we obtain
\begin{align} \label{equ:einq}
    \norm{\bm e^{(k+1)}}_2 \leq \bigl(1-\lambda_{\min}(\bm H)\bigr) \norm{\bm e^{(k)}}_2 + \norm{\bm r( \bm e^{(k)})}_2 . 
\end{align}
The conclusion of lemma follows from the fact that
\begin{align*}
    \norm{\bm e^{(k)}}_2 = \norm{\P_{\bar{\Omega}} \bigl( \bm E^{(k)} \bigr)}_F = \norm{\bm E^{(k)}}_F
\end{align*}
and
\begin{align*}
    \norm{\bm r( \bm e^{(k)})}_2 \leq \norm{\bm R(\bm E^{(k)})}_F \leq \frac{c_1}{\sigma_r} \norm{\bm E^{(k)}}_F^2 .
\end{align*}

\ifCLASSOPTIONcaptionsoff
  \newpage
\fi

% \newpage



%\bibliographystyle{IEEEtran}
\bibliographystyle{IEEEtran}
\bibliography{IEEEabrv,refs}

\vfill


\clearpage
\pagenumbering{arabic}
% \newgeometry{onecolumn,textwidth=\textwidth}
\twocolumn[{\Large \bf Supplementary Material - ``On Asymptotic Linear Convergence Rate of Iterative Hard Thresholding for Matrix Completion'', Trung~Vu, Evgenia~Chunikhina, and Raviv~Raich}\\ \\]

% \renewcommand{\appendixname}{Section}
% \appendices

% \section*{Proof of Example~\ref{eg:concentrated}}
% \label{appdx:eg}

\section*{The first case in Example~\ref{eg:concentrated}}
Using the same argument as in Lemma~5.3 in \cite{yaskov2014universality}, we can replace the complex matrix in (\ref{equ:concentrated}) by a real PSD matrix and prove the following lemma:
\begin{lemma} \label{lem:haar}
Let $\bm a = [a_1,\ldots,a_{qn}]^{\topnew}$ is a random vector with $i.i.d$ entries, where $a_i \sim \mathcal{N}(0,1/n)$. 
Then for any sequence of $qn \times qn$ PSD matrices $\bm M_{qn}$ with uniformly bounded spectral norms $\norm{\bm M_{qn}}_2$, we have
\begin{align*}
    \bigl( \bm a^{\topnew} \bm M_{qn} \bm a - \frac{1}{n} \tr(\bm M_{qn}) \bigr) \overset{\text{p}}{\to} 0 \text{ as } n \to \infty .
\end{align*}
\end{lemma}

\begin{proof}
To simplify our notation, let us denote the $(i,j)$-th entry of $\bm M_{qn}$ by $M_{ij}$ and $\delta_{ij}$ is the indicator of the event $i=j$.
Since $a_i$ are $i.i.d$ normally distributed, we have
\begin{align*}
    &{\mathbb E}[a_i] = 0, \quad {\mathbb E}[a_i a_j] = \delta_{ij} \frac{1}{n}, \\
    &{\mathbb E}[a_i a_j a_k a_l] = (\delta_{ij} \delta_{kl} + \delta_{ik} \delta_{jl} + \delta_{il} \delta_{jk}) \frac{1}{n^2} , \numberthis \label{equ:normalE}
\end{align*}
for any indices $1 \leq i,j,k,l \leq n$. In order to prove $\bigl( \bm a^{\topnew} \bm M_{qn} \bm a - \frac{1}{n} \tr(\bm M_{qn}) \bigr) \overset{\text{p}}{\to} 0$, it is sufficient to show that 
\begin{align*}
\begin{cases}
    {\mathbb E}[\bm a^{\topnew} \bm M_{qn} \bm a] = \frac{1}{n} \tr(\bm M_{qn}) , \\
    \Var(\bm a^{\topnew} \bm M_{qn} \bm a) \to 0 \text{ as } n \to \infty .
\end{cases}
\end{align*}
First, by the linearity of expectation, we have
\begin{align*}
    {\mathbb E}[\bm a^{\topnew} \bm M_{qn} \bm a] &= {\mathbb E} \Bigl[\sum_{i,j} M_{ij} a_i a_j \Bigr] \\
    &= \sum_{i,j} M_{ij} {\mathbb E} [a_i a_j] \\
    &= \sum_{i,j} M_{ij} \delta_{ij} \frac{1}{n} \\
    &= \frac{1}{n} \sum_{i=1}^{qn} M_{ii} \\
    &= \frac{1}{n} \tr(\bm M_{qn}) . \numberthis \label{equ:a_ijkl}
\end{align*}
Second, by rewriting the variance of the summation $\sum_{i,j} M_{ij} a_i a_j$ in terms of the sum of covariances, we obtain
\begin{align*}
    \Var(\bm a^{\topnew} \bm M_{qn} \bm a) &= \Var \Bigl( \sum_{i,j} M_{ij} a_i a_j \Bigr) \\
    &= \sum_{i,j,k,l} \Cov (M_{ij} a_i a_j, M_{kl} a_k a_l) . \numberthis \label{equ:cov_sum}
\end{align*}
Using the formula
\begin{align} \label{equ:cov}
    \Cov(X,Y) = {\mathbb E}[XY] - {\mathbb E}[X] {\mathbb E}[Y] ,
\end{align}
and the linearity of expectation, (\ref{equ:cov_sum}) can be represented as
\begin{align*}
    \Var(&\bm a^{\topnew} \bm M_{qn} \bm a) \\
    &= \sum_{i,j,k,l} M_{ij} M_{kl} \Bigl( {\mathbb E}[a_i a_j a_k a_l] - {\mathbb E}[a_i a_j] {\mathbb E}[a_k a_l] \Bigr) \\
    &= \sum_{i,j,k,l} M_{ij} M_{kl} \bigl( \delta_{ik} \delta_{jl} + \delta_{il} \delta_{jk} \bigr) \frac{1}{n^2} \\
    &= \frac{2}{n^2} \sum_{i,j} M_{ij}^2 \\
    &= \frac{2}{n^2} \norm{\bm M_{qn}}_F^2 . \numberthis \label{equ:var_aMa}
\end{align*}
Since $\bm M_{qn}$ is PSD and has bounded spectral norm, all of its eigenvalues are bounded by $0 \leq \lambda_i(\bm M_{qn}) \leq C$, for some constant $C$, and hence,
\begin{align*}
    \norm{\bm M_{qn}}_F^2 = \sum_{i=1}^{qn} \lambda_i^2(\bm M_{qn}) \leq qn C^2 .
\end{align*}
Thus, substituting back into (\ref{equ:var_aMa}) yields
\begin{align*}
    \Var(\bm a^{\topnew} \bm M_{qn} \bm a) \leq \frac{2}{n^2} qn C^2 \to 0 \text{ as } n \to \infty .
\end{align*}
This completes our proof of the lemma.
\end{proof}

\begin{figure*}
    \centering
    \begin{subfigure}[b]{0.45\textwidth}
        \centering
        \includegraphics[width=\textwidth]{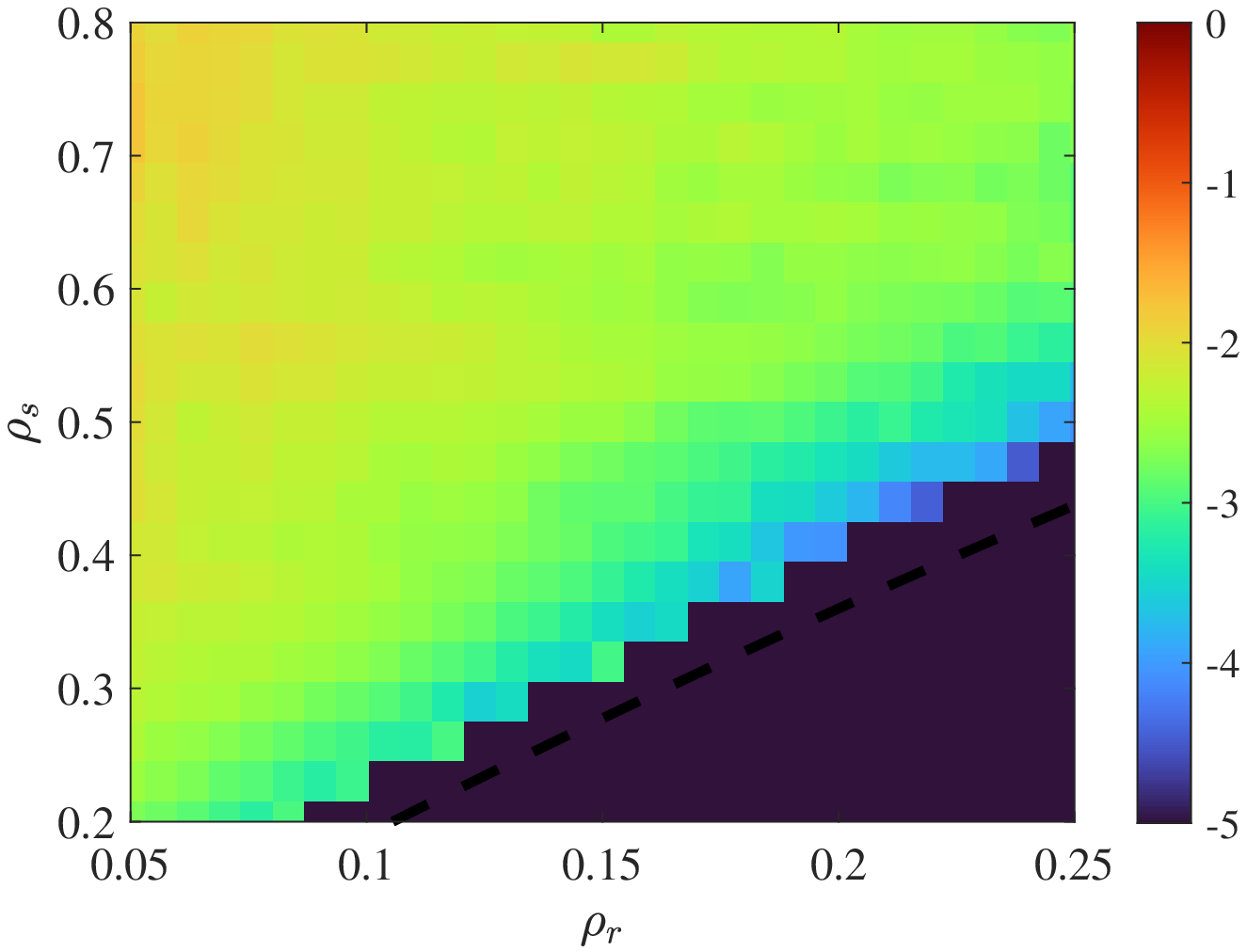}
        \caption{$n_1=500, n_2=400$}
    \end{subfigure}
    \hfill
    \begin{subfigure}[b]{0.45\textwidth}
        \centering
        \includegraphics[width=\textwidth]{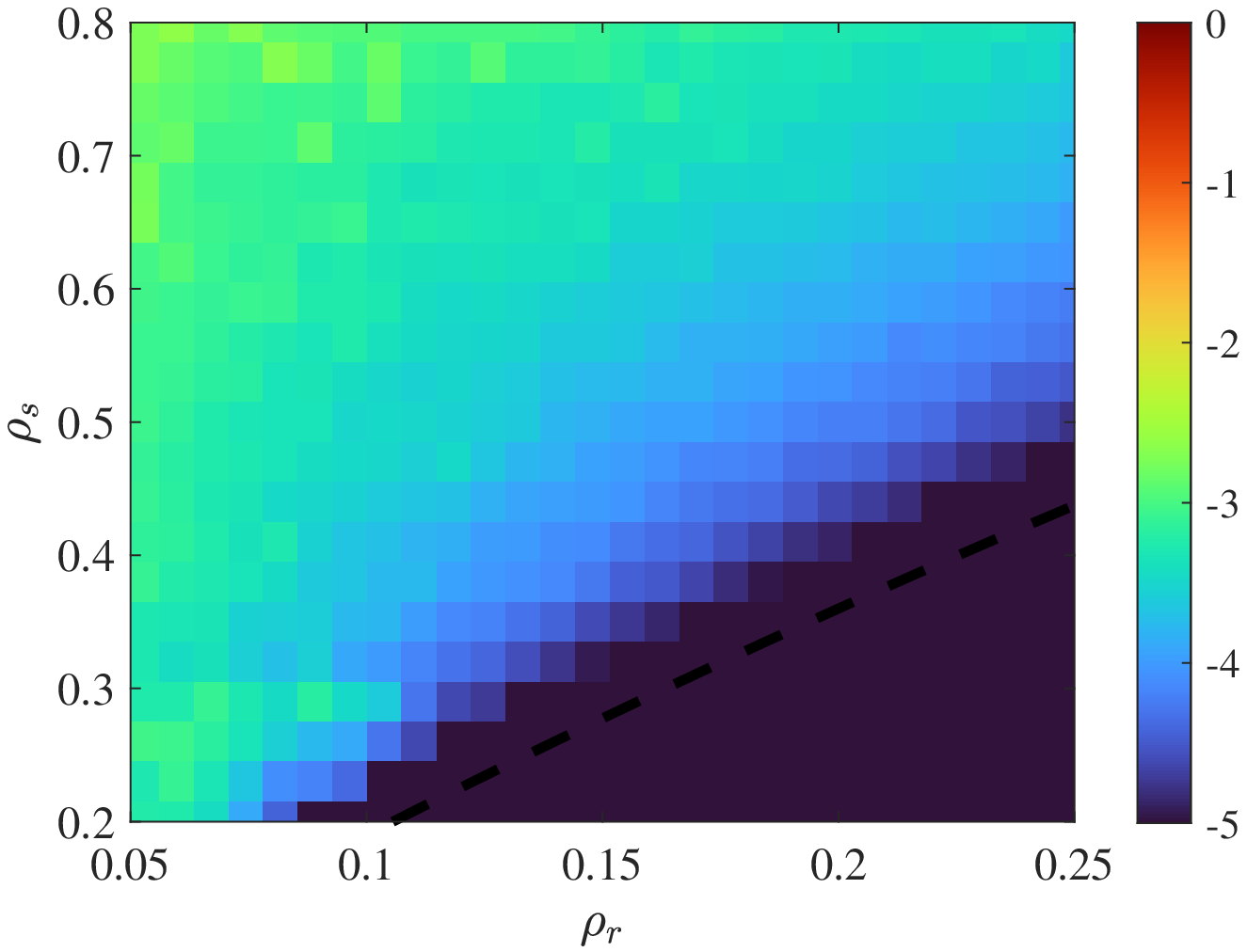}
        \caption{$n_1=1200, n_2=1000$}
    \end{subfigure}
    \hfill
    \caption{The coefficient of variation (on a log10 scale) of the empirical rate shown in Fig.~\ref{fig:asymp}-(a) and (b), respectively. In each plot, the black dashed line corresponds to the boundary line $1-\rho_s = (1-\rho_r)^2$ and the black region on the bottom-right corner corresponds to the settings where no linear convergence is observed (i.e., the empirical rate is set to $1$). The darker color in the right plot demonstrates the increasing concentration of the empirical rate as a random variable when the dimensions grow larger. It is also interesting to note that the variability in relation to the mean decreases as it approaches the boundary line (i.e., from the top-left corner to the bottom-right corner).}
    \label{fig:var}
\end{figure*}

\section*{The second case in Example~\ref{eg:concentrated}}
Similarly, we consider the following lemma:
\begin{lemma} \label{lem:kron}
Let $\bm b = [b_1,\ldots,b_{qn}]$ and $\bm c = [c_1,\ldots,c_{qn}]$ are random vectors with $i.i.d$ entries, where $b_i, c_j \sim \mathcal{N}(0,1/n)$. Denote $m=n^2$, $k=q^2$ and $\bm a = \bm b \otimes \bm c$. 
Then for any sequence of $km \times km$ PSD matrices $\bm M_{km}$ with uniformly bounded spectral norms $\norm{\bm M_{km}}_2$, we have
\begin{align*}
    \bigl( \bm a^{\topnew} \bm M_{km} \bm a - \frac{1}{m} \tr(\bm M_{km}) \bigr) \overset{\text{p}}{\to} 0 \text{ as } n \to \infty .
\end{align*}
\end{lemma}

\begin{proof}
Denote $\bm M_{[ij]}$ is the $(i,j)$-th $qn \times qn$ block of $\bm M_{km}$. Then it is straightforward to verify that
\begin{align*}
    \bm a^{\topnew} \bm M_{km} \bm a = \sum_{i,j} b_i (\bm c^{\topnew} \bm M_{[ij]} \bm c) b_j .
\end{align*}
In order to prove $\bigl( \bm a^{\topnew} \bm M_{km} \bm a - \frac{1}{m} \tr(\bm M_{km}) \bigr) \overset{\text{p}}{\to} 0$, it is sufficient to show that 
\begin{align*}
\begin{cases}
    {\mathbb E}[\bm a^{\topnew} \bm M_{km} \bm a] = \frac{1}{m} \tr(\bm M_{km}) , \\
    \Var(\bm a^{\topnew} \bm M_{km} \bm a) \to 0 \text{ as } n \to \infty .
\end{cases}
\end{align*}
First, we use the linearity of expectation to obtain
\begin{align*}
    {\mathbb E}[\bm a^{\topnew} \bm M_{km} \bm a] &= {\mathbb E} \Bigl[ \sum_{i,j} b_i (\bm c^{\topnew} \bm M_{[ij]} \bm c) b_j \Bigr] \\
    &= \sum_{i,j} {\mathbb E} [b_i b_j] {\mathbb E}[\bm c^{\topnew} \bm M_{[ij]} \bm c] .
\end{align*}
From (\ref{equ:a_ijkl}) and Lemma~\ref{lem:haar}, the last equation is equivalent to
\begin{align*}
    {\mathbb E}[\bm a^{\topnew} \bm M_{km} \bm a] &= \sum_{i,j} \delta_{ij} \frac{1}{n} \cdot \frac{1}{n} \tr(\bm M_{[ij]}) \\
    &= \frac{1}{m} \tr (\bm M_{km}) .
\end{align*}
Second, we have
\begin{align*}
    \Var &(\bm a^{\topnew} \bm M_{km} \bm a) = \Var \bigl(\sum_{i,j} b_i (\bm c^{\topnew} \bm M_{[ij]} \bm c) b_j \bigr) \\
    &= \sum_{i,j,k,l} \Cov \bigl(b_i (\bm c^{\topnew} \bm M_{[ij]} \bm c) b_j , b_k (\bm c^{\topnew} \bm M_{[kl]} \bm c) b_l \bigr) . \numberthis \label{equ:sum_var}
\end{align*}
From (\ref{equ:cov}), each covariance on the RHS of (\ref{equ:sum_var}) can be represented as
\begin{align*}
    \Cov &\bigl(b_i (\bm c^{\topnew} \bm M_{[ij]} \bm c) b_j , b_k (\bm c^{\topnew} \bm M_{[kl]} \bm c) b_l \bigr) \\
    &= {\mathbb E}[b_i b_j b_k b_l] \cdot {\mathbb E}[\bm c^{\topnew} \bm M_{[ij]} \bm c \cdot \bm c^{\topnew} \bm M_{[kl]} \bm c] \\
    &\quad - {\mathbb E}[b_i b_j] \cdot {\mathbb E}[b_k b_l] \cdot {\mathbb E}[\bm c^{\topnew} \bm M_{[ij]} \bm c] \cdot {\mathbb E}[ \bm c^{\topnew} \bm M_{[kl]} \bm c] . \numberthis \label{equ:cov_ijkl}
\end{align*}
\begin{lemma} \label{lem:cAB}
Let $\bm P$ and $\bm Q$ be matrices in $\R^{qn \times qn}$. Then 
\begin{align*}
    {\mathbb E}[\bm c^{\topnew} \bm P \bm c \cdot \bm c^{\topnew} \bm Q \bm c] = \frac{\tr(\bm P) \tr(\bm Q) + \tr(\bm P \bm Q^{\topnew}) + \tr(\bm P \bm Q)}{n^2} .
\end{align*}
\end{lemma}
\noindent The proof of Lemma~\ref{lem:cAB} is straightforward from (\ref{equ:normalE}) and is omitted in this manuscript.  From Lemma~\ref{lem:cAB} and (\ref{equ:normalE}), we can simplify (\ref{equ:cov_ijkl}) as
\begin{align*}
    \Cov &\bigl(b_i (\bm c^{\topnew} \bm M_{[ij]} \bm c) b_j , b_k (\bm c^{\topnew} \bm M_{[kl]} \bm c) b_l \bigr) \\
    &= \frac{1}{n^4} \Bigl( \tr(\bm M_{[ij]} \bm M_{[kl]}) + \tr(\bm M_{[ij]} \bm M_{[kl]}^{\topnew}) \\
    &\qquad \qquad + {\tr}^2(\bm M_{[ij]}) + \tr(\bm M_{[ij]}^2) + \tr(\bm M_{[ij]} \bm M_{[ij]}^{\topnew}) \\
    &\qquad \qquad + \tr(\bm M_{[ij]}) \tr(\bm M_{[ij]}^{\topnew}) + \tr(\bm M_{[ij]}^2) \Bigr) .
\end{align*}
Substituting the last equation back into (\ref{equ:sum_var}) yields
\begin{align*}
    &\Var (\bm a^{\topnew} \bm M_{km} \bm a) = \frac{2}{n^4} \Bigl( \sum_{i,j} {\tr}^2(\bm M_{[ij]}) + \sum_{i,j} {\tr}(\bm M_{[ii]} \bm M_{[jj]}) \\
    &\qquad \qquad + \sum_{i,j} {\tr}(\bm M_{[ij]}^{\topnew} \bm M_{[jj]}) + \sum_{i,j} {\tr}(\bm M_{[ij]}^2) \Bigr) . \numberthis\label{equ:var_tr}
\end{align*}
Next, we bound each term on the RHS of (\ref{equ:var_tr}). To that end, we utilize the following lemma:
\begin{lemma} \label{lem:trA}
For any matrices $\bm A, \bm B \in \R^{n \times n}$, it holds that
\begin{enumerate}
    \item $\norm{\bm A}_F \leq \sqrt{n} \norm{\bm A}_2$,
    \item $\tr^2(\bm A) \leq n \norm{\bm A}_F^2$,
    \item $\tr(\bm A^{\topnew} \bm B) \leq \norm{\bm A}_F \norm{\bm B}_F \leq n \norm{\bm A}_2 \norm{\bm B}_2$,
    \item $\tr(\bm A^2) \leq \norm{\bm A}_F^2 = \tr(\bm A^{\topnew} \bm A)$.
\end{enumerate}
\end{lemma}
\noindent The proof of Lemma~\ref{lem:trA} can be found in \cite{meyer2000matrix} - Chapter~5. Applying Lemma~\ref{lem:trA} with the blocks of size $qn \times qn$, we obtain
\begin{align*}
    \sum_{i,j} {\tr}^2(\bm M_{[ij]}) &\leq \sum_{i,j} qn \norm{\bm M_{[ij]}}_F^2 = qn \norm{\bm M}_F^2 \\
    &\leq (qn)^3 \norm{\bm M}_2 \leq C (qn)^3 ,
\end{align*}
\begin{align*}
    \sum_{i,j} {\tr}(\bm M_{[ii]} \bm M_{[jj]}) &\leq \sum_{i,j} qn \norm{\bm M_{[ii]}}_2 \norm{\bm M_{[jj]}}_2 \\
    &\leq \sum_{i,j} qn \norm{\bm M}_2 \norm{\bm M}_2 = C^2 (qn)^3 ,
\end{align*}
\begin{align*}
    \sum_{i,j} {\tr}(\bm M_{[ij]}^{\topnew} \bm M_{[jj]}) &= \sum_{i,j} \norm{\bm M_{[ij]}}_F^2 = \norm{\bm M}_F^2 \leq C (qn)^2 ,
\end{align*}
\begin{align*}
    \sum_{i,j} {\tr}(\bm M_{[ij]}^2) &\leq \sum_{i,j} \norm{\bm M_{[ij]}}_F^2 = \norm{\bm M}_F^2 \leq C (qn)^2 .
\end{align*}
Therefore, (\ref{equ:var_tr}) can be bounded as
\begin{align*}
    \Var (\bm a^{\topnew} \bm M_{km} \bm a) &\leq \frac{2}{n^4} (C(qn)^3 +  C^2(qn)^3 + 2 C (qn)^2) .
\end{align*}
The conclusion of the lemma follows by the fact that the RHS of the last equation which approaches $0$ as $n \to \infty$.
\end{proof}

\end{document}